\documentclass[final,sort&compress,3p,times,12pt]{elsarticle}
\usepackage{amsmath}
\allowdisplaybreaks[4]
\usepackage{amsmath}
\usepackage{amsfonts}
\usepackage{amssymb}
\usepackage{dsfont}
\usepackage{amsmath}
\usepackage{amsthm}
\usepackage{graphicx}
\usepackage{grffile}
\usepackage{subfigure}
\usepackage[justification=centering]{caption}
\usepackage{color}
\usepackage{array}
\usepackage{mathrsfs}
\usepackage{multirow}
\newtheorem{example}{Example}
\newtheorem{Theo}{Theorem}
\newtheorem{Assu}{Assumption}
\newtheorem{Rem}{Remark}
\newtheorem{Lem}{Lemma}

\newtheorem{Coro}{Corollary}

\numberwithin{equation}{section}
 \numberwithin{Lem}{section}
 \numberwithin{Defi}{section}
 \numberwithin{Theo}{section}
 \numberwithin{Pro}{section}
 \numberwithin{Rem}{section}
  \numberwithin{Coro}{section}
  \numberwithin{Fig}{section}
\def\RR{\mathbb{R}}

\def\EE{\mathbb{E}}

\def\TT{\mathbf{T}}

\def\al{{\alpha}}

\numberwithin{equation}{section}

\begin{document}

\begin{frontmatter}

\title{Numerical methods for stochastic Volterra integral equations with weakly singular kernels$^\star$ \tnotetext[1]{This work was supported by National Natural Science Foundation of China (No. 11771163),    China Scholarship Council under 201906160066, and by
 Natural Sciences and Engineering Research Council of Canada (NSERC) discovery grant.
}}
% \tnotetext[label1]{}
\author{Min Li$^a$,\quad Chengming Huang$^{a,b}$, \quad Yaozhong Hu$^c*$}
%\author{Min Li$^a$,\quad Chengming Huang$^{a,b,*}$,\quad Peng Hu$^c$ }
\cortext[cor1]{Corresponding author.\\
\emph{Email addresses}: \texttt{liminmaths@163.com} (M. Li),
\texttt{chengming\_huang@hotmail.com} (C. Huang),\\
\texttt{yaozhong@ualberta.ca} (Y. Hu$^*$)}\ \\
%\emph{Email addresses}: \texttt{liminmaths@163.com} (M. Li),
%\texttt{chengming\_huang@hotmail.com} (C. Huang),
%\texttt{yaozhong@ualberta.ca} (Y. Hu$^*$)}.\\
%\texttt{yuechao2005@163.com} (Y. Chao)
 \address{$^a$School of Mathematics and Statistics, Huazhong University of Science and Technology, Wuhan 430074,
 China}
 \address{$^b$Hubei Key Laboratory of Engineering Modelling and Scientific Computing, Huazhong University of Science
 and Technology, Wuhan 430074, China}
 \address{$^c$Department of Mathematical and Statistical Sciences, University of Alberta,  Edmonton T6G 2G1, Canada~ }
\date{}

\begin{abstract}
 In this paper, we first establish the existence, uniqueness and H\"older continuity of the solution to
stochastic Volterra integral equations with  weakly singular kernels.
Then, we propose a  $\theta$-Euler-Maruyama  scheme  and a Milstein
scheme to solve  the equations numerically  and  we obtain
 the strong  rates of convergence     for both schemes
 in $L^{p}$ norm for any $p\geq 1$.    For the $\theta$-Euler-Maruyama  scheme the rate is
  $\min\{1-\alpha,\frac{1}{2}-\beta\}~
% (0<\alpha<1, 0< \beta<\frac{1}{2})
 $ and for the Milstein scheme the rate is
 $\min\{1-\alpha,1-2\beta\}$ when $\alpha\neq \frac 12$,  where $(0<\alpha<1, 0< \beta<\frac{1}{2})$.
  These  results on the rates of
 convergence  are   significantly different from that of the similar schemes for the
stochastic Volterra integral equations with regular  kernels.
The difficulty to obtain our results is the lack of It\^o formula for the equations. To get around of this difficulty we use instead the Taylor formula   and then carry a
sophisticated  analysis on   the equation the solution satisfies.
%  Finally, some numerical experiments are included to support
%  our theoretical results.
\end{abstract}
\end{frontmatter}

% main text
\section{Introduction}
Let $(\Omega, \mathscr{F},\mathbb{P})$  be a complete probability space   with a  filtration
  $\{\mathscr{F}_{t}\}_{0\le t\le 1}$ satisfying the  usual condition.
  The expectation on this space is denoted by $\EE$.
 Let  $W(t):=(W_{1}(t),
\cdots, W_{m}(t))^{T}$, $0\le t\le 1$,  be  an $m$-dimensional Wiener process defined on the probability space $(\Omega, \mathscr{F},\mathbb{P})$  adapted to the filtration $\mathscr{F}_{t}$.
 Assume  $a:[0, 1]\times \mathbb{R}^{d}\rightarrow \mathbb{R}^{d}$ and
$b:[0, 1]\times \mathbb{R}^{d}\rightarrow \mathbb{R}^{d\times m}$
satisfy some conditions that we shall specify in next section.
In this paper we shall
consider the numerical approximation of the following $d$-dimensional    stochastic Volterra integral equations (SVIEs)  with weakly singular kernel
\begin{equation}\label{a1}
X(t)=X_{0}+\int_{0}^{t}(t-s)^{-\alpha}a(X(s))ds+\int_{0}^{t}(t-s)^{-\beta}b(X(s))dW_{s},~~~t\in [0,1]\,,
\end{equation}
 where   $\alpha\in (0,1)$, $\beta\in (0,\frac{1}{2})$ are two given positive numbers    and the initial condition can be
 random and satisfies $\mathbb{E}|X_{0}|^{p}<\infty$ for any $p\geq 1$.
 We consider the interval  $[0, 1]$ for notational simplicity.
 It is easy to extend all the results of this paper to equation on any finite interval $[0, T]$ instead of $[0, 1]$.
%   denotes the exact solution to \eqref{a1},
% and
When $\alpha=\beta=0$, the above stochastic differential equations (SDEs), including their numerical schemes, have been very well-studied. Many monographs
can be found so that  we are not going to give any references
here.  Relatively, the  singular  Volterra integral equations  of the above   form   have been less studied.
% although  they play   key roles
%  in many phenomena in engineering, physics, control theory.
%%   Amongst huge literature in this field we only mention a few of them.
%  (see \cite{Brunner2004Volterra,Brunner1986numerical,
%  Brunner2017Volterra,Linz1985Analytical}).
%   As well known, the well-posedness were well established for SDEs (cf.
%  \cite{Mao2008SDEs_application,Oksendal2003SDEs,Gard1988Introduction_SDEs}) and
%   SVIEs (cf. \cite{Berger1980VolterraI,Berger1980VolterraII,Ito1979Stochastic_Volterra,Zhang2008Euler_Volterra}).
We mention some
  existence and uniqueness  results
  under  the  (global)  Lipschitz
  condition and the linear growth condition (see \cite{Son2018Caputo,Anh2019variation,Abouagwa2019Doob,Abouagwa2019Levy}).

%In the past several decades, a wide range of numerical methods were constructed for SDEs (see
%\cite{Kloeden1992Numerical,Milstein1995Numerical,Milstein2004Stochastic_numerics}
%and references therein). However,
 When $(t-s)^{-\alpha}$ and
$(t-s)^{-\beta}$   are replaced by some nice  functions, the numerical schemes  of (regular) SVIEs have received attention only quite
recently.
Tudor  \cite{Tudor1995Approximation}   studied the strong convergence of one-step numerical approximations
 for It\^{o}-Volterra equations, and he obtained   the rate of convergence in  the mean-square  sense ($L_p $ when $p=2$ in our terminology here).
Wen and Zhang \cite{Wen2011Improved_rectangular} analysed an improved variant of the rectangular method
 for stochastic Volterra equation, and the order of convergence was shown to be $1.0$. Subsequently, Wang \cite{Wang2017Approximate}
approximated  the  solutions to SVIEs by means of solutions to a class of SDEs  and he studied
two numerical methods: stochastic theta method and splitting method.
  Xiao et al. \cite{Xiao2018collocation} introduced a split-step collocation method for SVIEs,
   and the method was proved to be convergent with order $0.5$.  Most relevant to our work is the work of
  Liang et al. \cite{Liang2017superconvergence}  who
 found that Euler-Maruyama (EM) method can achieve a  superconvergence of order $1.0$
 if the kernel function  in diffusion term
  satisfies certain boundary condition. More recently,
 for the Euler scheme for more general class of equations, such as SVIEs with delay, stochastic Volterra
 integro-differential equations and stochastic fractional integro-differential equations, we refer to
 \cite{Gao2019delay,Yang2019Theoretical,Zhang2020generalized,Zhang2020non-globally,Dai2019fractional}.

To  the best of our knowledge, there have been not yet  numerical schemes  for SVIEs
with weakly singular kernel like the form in \eqref{a1}. The difficulty is probably    the singularity of the integrand kernel:  In this case
the powerful and necessary tool of It\^{o} formula commonly used previously
 does not exist for SVIEs with singular kernel.
  In this work we  fill  this gap by providing
 strong convergence rates of $\theta$-Euler-Maruyama  scheme and  Milstein scheme.
 Our results for the numerical parts are summarized as follows.  For Euler scheme ($Y_n$ obtained from \eqref{a21})  we shall prove that for any $p\geq 1$
\begin{align*}
\max_{n\in \{0,\cdots,N\}}\|X(t_{n})-Y_{n}\|_{L^{p}(\Omega;\mathbb{R}^{d})}\leq Ch^{\min\left(\frac{1}{2}-\beta,~ 1-\alpha\right)},
\end{align*}
where  $h$ is the mesh size.
 And for Milstein type scheme (the $Z_n$ given by \eqref{a211}) we shall prove the estimate:
\begin{align*}
\max_{n\in \{0,\cdots,N\}}\|X(t_{n})-Z_{n}\|_{L^{p}(\Omega;\mathbb{R}^{d})}\leq \left\{
                                                     \begin{array}{ll}
                                                     Ch^{\min\{1-\alpha, 1-2\beta\}}, & \hbox{$\alpha\neq\frac{1}{2}$;} \\
                                                     %C\max\big{\{}h^{1-2\beta},h^{2(1-\beta)}\ln(1/h)\big{\}}, & \hbox{$\alpha=\frac{1}{2}$;} \\
                                                     C\max\{h^{\min\{\frac{1}{2}, 1-2\beta\}},h^{(1-\beta)}(\ln(\frac{1}{h}))^{1/2}\}, & \hbox{$\alpha=\frac{1}{2}$.}
                                                     \end{array}
                                                          \right.
\end{align*}
Since we can no longer use the It\^o-Taylor formula for the sultion of the equation, we shall use only the Taylor formula combined with  the techniques of classical fractional calculus,
 and discrete and continuous typed Gronwall inequalities with weakly singular kernels.

%   Recently, Beyn et al. \cite{Beyn2016C-stability,Beyn2017Milstein}
%  investigated convergence of  projected EM method, projected Milstein  method for SODEs by using the notions of
%  C-stability and B-consistency, which can avoid those processes on the discrete time level.
    %In this way,

%Compared with implicit methods, % for SDEs,
%the explicit Euler methods process simpler algebraic structure, and can reach strong convergence of
%order $1/2$ with cheaper computational cost. However, Hutzenthaler et. al
%\cite{Hutzenthaler2011divergence} proved that strong and weak divergence in finite time of the
%explicit Euler methods for SODEs with superlinearly growing coefficients. Subsequently, some
%modified  Euler methods, such as tamed and truncated techniques, were constructed to solve the
%nonlinear SODEs (see \cite{Mao2016trancated,Hutzenthaler2012explicit}). The main goal of this paper
%is to generalize the projected EM method for SDDEs with superlinearly growth conditions, so that
%the convergence analysis can be simplified significantly.

%An outline
The remaining part  of the paper is organized as follows. In Section $2$, some assumptions and preliminaries are
introduced.  The main results of the paper on the existence, uniqueness, and H\"older continuity of the solution and  the strong convergence rate results are  stated. When the kernel are singular it seems that the well-posedness of the equation has not been studied yet.  Section $3$ studies  the existence, uniqueness of the exact solution of the SVIEs with singular kernel. On the other hand, to obtain the rates of convergence of our schemes we also need to use th H\"older continuity of the solution.
All of these are done in Section $3$.   %the convergence results of $\theta$-EM method.
In Section $4$, we present  a proof of the convergence results of $\theta$ Euler-Maruyama  scheme.
 In Section $5$, we present a proof of the convergence results of Milstein-type scheme.
 In Section $6$,  we present  some numerical simulations
  to support our theoretical results.  %Finally, some conclusions are drawn in the last section.

\section{Preliminaries  and main results }
We  need to use the following generalized (discrete and continuous types) Gronwall inequalities with weakly singular kernels, whose proofs
 can be found in \cite{Brunner2004Volterra}.
\begin{Lem}\label{Lemma 3.2}
Let $b>0$ be a positive number.
 If the non-negative sequence $\{H_{n}\}$ satisfies the inequality
 $$
H_{n}\leq \pi_{n}+b\sum_{l=0}^{n-1}(n-l)^{-\gamma}H_{l},~~~0\leq n\leq N,
$$
%$$
%H_{n}\leq \pi_{n}+Mh^{1-\gamma}\sum_{l=0}^{n-1}(n-l)^{-\gamma}H_{l},~~~0\leq n\leq N,
%$$
where the sequence $\{\pi_{n}\}$ is non-negative,   and $0<\gamma<1$, then
$$
H_{n}\leq E_{1-\gamma}( \Gamma(1-\gamma) n ^{1-\gamma} b)\pi_{n},
$$
where  $\Gamma(a)=\int_0^\infty e^{-s}  s^{a-1}ds$, $a>0$ is the
Euler Gamma function,   and
\[
E_{a}(x):=\sum_{k=0}^\infty
\frac{1}{\Gamma(a k+1)} x^k\,, \quad a>0
\]
is the Mittag-Leffler function   of $x\in  \RR$  (cf. \cite{Gorenfl2014Mittag}).
\end{Lem}

\begin{Lem}\label{Lemma 3.4}
Let $I:=[0, 1]$ and assume that\\
(i)~$g\in C(I)$ (the set of real valued continuous functions on $I$)
and $g$ is non-decreasing on $I$.\\
(ii)~the continuous, non-negative function $H$ satisfies the inequality
$$
H(t)\leq g(t)+b\int_{0}^{t}(t-s)^{-\gamma}H(s)ds
$$
for constant $b>0$ and $0<\gamma<1$.  Then
$$
H(t)\leq E_{1-\gamma}(\Gamma(1-\gamma)t^{1-\gamma} b)g(t),~~~t\in I.
$$
\end{Lem}

The assumptions that we are going to make about the coefficients in
our main equation \eqref{a1}  are summarized as follows.
\begin{Assu}\label{assu1}
Assume that\\
(i) there exists positive constant $\hat{L}$ such that
$$
|a(x)-a(y)|\vee |b(x)-b(y)|\leq \hat{L}|x-y|;
$$
(ii) the function $a$ and $b$ satisfy the linear growth condition
$$
|a(x)| \leq \hat{L}(1+|x|),~~|b(x)| \leq \hat{L}(1+|x|).
$$
\end{Assu}

\begin{Assu}\label{assu2}
Assume that there exists positive constant $\hat{L}$ such that the derivatives of function $a$ satisfies
$$
|\nabla a (x)-\nabla a(y)|\leq \hat{L}|x-y|;
$$
\end{Assu}

We now show that Assumption \ref{assu1} is sufficient to ensure the existence and uniqueness of the solution to equation \eqref{a1}.
\begin{Theo}\label{Theorem2.1}
Assume that the coefficients  $a$ and $b$ satisfy Assumption \ref{assu1}. Then there is a unique solution $X(t)$ to \eqref{a1}, and
the solution satisfies that for any $p\ge 1$,
\begin{equation}
\sup_{0\le t\le 1} \mathbb{E}|X(t)|^p\leq C_p, \label{e.2.1}
\end{equation}
where and throughout the remaining part of the paper we denote by $C$
(or $C_p$)  a generic  constant  (independent of $h$)  which may have different
values in different places.
\end{Theo}

We also need the H\"older continuity in the p-th ($p\geq 1$) moment
of the exact solution.
\begin{Theo}\label{Theorem2.2}
Let the assumptions 1 and 2 are satisfied.
Denote $\gamma(\alpha, \beta)=\min\{\frac12 - \beta,~  1-\alpha\}$.
Then for the solution
$X$ to
\eqref{a1} we have   for any $p\ge 1$,
\begin{equation}
\mathbb{E}|X(t)-X(r)|^p\leq C_p |t-r|^{p\gamma(\alpha, \beta)}\,,\quad \forall \ 0\le r\le t\le 1\,.
\label{e.2.4}
\end{equation}
\end{Theo}

 The proofs of the above two  theorems
 (Theorem \ref{Theorem2.1} and \ref{Theorem2.2})  are  given in Section \ref{section3}.

%The next stability lemma plays an important role in the convergence analysis.
%\begin{Lem}
%If $(\Psi, h, \xi)$ is stochastically C-stable one-step method with constants $C$, and $\eta \in (1,\infty)$,  then for every grid function $Z\in \mathscr{G}^{2}(\mathcal{T}_{h})$,
%\begin{align}\label{a7}
%\begin{split}
%\max_{n\in\{0\cdots N\}}&\|Z(t_{n})-X_{h}(t_{n})\|_{L^{2}(\Omega;\mathbb{R}^{d})}^{2}\\
%\leq& e^{2(1+C(1+h))T}\bigg{(}\sum_{i=1}^{M}\big{\|}Z(t_{i-M})-\xi(t_{i-M})\big{\|}_{L^{2}(\Omega;\mathbb{R}^{d})}^{2}+
%\|Z(t_{0})-X_{h}(t_{0})\|_{L^{2}(\Omega;\mathbb{R}^{d})}^{2}\\
%&+\sum_{i=1}^{N}(1+h^{-1})\big{\|}\mathbb{E}\big[Z(r )-\Psi(Z(t_{i-1}),Z(t_{i-M}),h)|\mathscr{F}_{t_{i-1}}]\big{\|}_{L^{2}(\Omega;\mathbb{R}^{d})}^{2}\\
%&+C_{\eta}\sum_{i=1}^{N}\big{\|}(id-\mathbb{E}[\cdot|\mathscr{F}_{t_{i-1}}])\big{(}Z(t_{i})-\Psi(Z(t_{i-1}),Z(t_{i-M}),h)\big{)}\big{\|}_{L^{2}(\Omega;\mathbb{R}^{d})}^{2}
% \bigg{)},
%\end{split}
%\end{align}
%where $Z(t_{i-M})$, $\xi(t_{i-M}), i=0, 1, \cdots, M$, are defined by $Z(t_{i}-\tau)$ and $\xi(t_{i}-\tau)$, respectively.
%\end{Lem}
%
%For the convergence analysis, we also need the following generalized discrete Gronwall inequality.

Let $h>0$ be the mesh size.
Throughout this paper, we consider only    uniform mesh on $[0, 1]$ by
$$
t_{i}=ih,~~i=0,1,\cdots,N,~~h=\frac{1}{N}.
$$
Denote  $\eta(s)=t_{i}$ for $t_{i}\leq s<t_{i+1}$
 and $\lfloor x\rfloor$
  is the greatest   integer less than or equal to $x$.
%Moreover, we follow the notation of the space of adapted and square integrable grid functions
%$$
%\mathcal{G}^{2}(\mathcal{T}_{h}):=\{Z:\mathcal{T}_{h}\times\Omega\rightarrow \mathbb{R}^{d}; Z(t_{n})\in L^{2}(\Omega,\mathscr{F}_{t_{n}}, \mathbb{P}; \mathbb{R}^{d})~ \text{for~ all}~ n=0, 1, \cdots, N\}.
%$$
%With the help of the preceding notations, we can give the definition of stochastic one-step methods.
%\begin{Defi}
%For every $t, t+h \in [0, T]$ and $Z\in L^{2}(\Omega,\mathscr{F}_{t},\mathbb{P};\mathbb{R}^{d})$, $\Phi$ satisfies the following measurability and integrability condition:
%$$
%\Phi(Z,t,h)\in L^{2}(\Omega,\mathscr{F}_{t+h}, \mathbb{P}; \mathbb{R}^{d}),
%$$
%if
%\begin{align*}
%&Y_{i}=\Phi(X_{h}(t_{i-1}),X_{i-M},h),~~1\leq i\leq N,\\
%&X_{i-M}=\xi(t_{i}-\tau),~~~i=0,\cdots,M,
%\end{align*}
%then we say grid function $X_{h}\in\mathcal{G}^{2}(\mathcal{T}_{h})$ is yielded by the
%$\text{stochastic one-step method}~ (\Phi, h,\xi)$.
%\end{Defi}
%$\Phi$ is referred as $\text{one step map}$ of the method .

%Taking one step projected EM method in \cite{Beyn2016C-stability} into account, we propose our new
%projected method $(\Phi, h, \xi)$ for SDDEs \eqref{a1} as follows
We first introduce the
following  $\theta$-Euler-Maruyama ($\theta$-EM )
   and Milstein-type schemes for SVIEs with weakly singular kernel respectively %\eqref{a1}
as follows
\begin{align}\label{a21}
\begin{split}
Y_{n+1}=&Y_{0}+\theta\sum_{i=0}^{n}\int_{t_{i}}^{t_{i+1}}(t_{n+1}-s)^{-\alpha}a(Y_{i+1})ds+(1-\theta)\sum_{i=0}^{n}\int_{t_{i}}^{t_{i+1}}
(t_{n+1}-s)^{-\alpha}a(Y_{i})ds\\
&+\sum_{i=0}^{n}\int_{t_{i}}^{t_{i+1}}(t_{n+1}-t_{i})^{-\beta}b(Y_{i})dW_{s},~~~n=0,1, \cdots, N-1,~~\theta\in [0,1],
\end{split}
\end{align}
and
\begin{align}\label{a211}
\begin{split}
Z_{n+1}=&Z_{0}+\sum_{i=0}^{n}\int_{t_{i}}^{t_{i+1}}(t_{n+1}-s)^{-\alpha}a(Z_{i})ds+
\sum_{i=0}^{n}\int_{t_{i}}^{t_{i+1}}(t_{n+1}-s)^{-\beta}b(Z_{i})dW_{s}\\
&+\sum_{i=0}^{n}\int_{t_{i}}^{t_{i+1}}(t_{n+1}-s)^{-\beta}b^{'}(Z_{i})\bigg{(}\sum_{l=0}^{i-1}\int_{t_{l}}^{t_{l+1}}[(s-r)^{-\beta}
-(t_{i}-r)^{-\beta}]b(Z_{l})dW_{r}\bigg{)}dW_{s}\\
&+\sum_{i=0}^{n}\int_{t_{i}}^{t_{i+1}}(t_{n+1}-s)^{-\beta}b^{'}(Z_{i})\bigg{(}\int_{t_{i}}^{s}(s-r)^{-\beta}b(Z_{i})dW_{r}\bigg{)}dW_{s},~~~n=0,1, \cdots, N-1,
\end{split}
\end{align}
 where $Y_{0}=Z_{0}=X_{0}$.

 The main results of this paper are the following strong rates  of convergence of
 the $\theta$ Euler-Maruyama scheme \eqref{a21}  and Milstein-type scheme \eqref{a211}  for
the solution of the SVIEs with weakly singular kernel \eqref{a1}.
We shall provide their  proofs   in Sections     \ref{section4} and \ref{section5}, respectively.
\begin{Theo}\label{theorem2.3}
If the functions $a$ and $b$ satisfy Assumption \ref{assu1}, then for any $p\geq 1$ %there exists a constant $C$ (independent of $h$) such that
\begin{align*}
\max_{n\in \{0,\cdots,N\}}\|X(t_{n})-Y_{n}\|_{L^{p}(\Omega;\mathbb{R}^{d})}\leq Ch^{\min\big{(}\frac{1}{2}-\beta,~ 1-\alpha\big{)}},
\end{align*}
where $X$ is the exact solution of \eqref{a1} and $Y_{n}$ is the numerical solution obtained from the $\theta$  Euler-Maruyama scheme dictated  by \eqref{a21}.
\end{Theo}

\begin{Theo}\label{t.2.4}
If the functions $a$ and $b$  and  their derivatives till the third order are bounded,
 and if the Assumption \ref{assu1} and Assumption \ref{assu2} are satisfied, then  for any $p\geq 1$ %there exists a constant $C$ (independent of $h$) such that
\begin{align*}
\max_{n\in \{0,\cdots,N\}}\|X(t_{n})-Z_{n}\|_{L^{p}(\Omega;\mathbb{R}^{d})}\leq \left\{
                                                     \begin{array}{ll}
                                                     Ch^{\min\{1-\alpha, 1-2\beta\}}, & \hbox{$\alpha\neq\frac{1}{2}$;} \\
                                                     %C\max\big{\{}h^{1-2\beta},h^{2(1-\beta)}\ln(1/h)\big{\}}, & \hbox{$\alpha=\frac{1}{2}$;} \\
                                                     C\max\{h^{\min\{1/2, 1-2\beta\}},h^{(1-\beta)}(\ln(\frac{1}{h}))^{1/2}\}, & \hbox{$\alpha=\frac{1}{2}$.}
                                                     \end{array}
                                                          \right.
\end{align*}
\end{Theo}

\begin{Rem}
When the singular parameter $\alpha=\beta=0$,   the above two theorems say that the $\theta$ Euler-Maruyama scheme  \eqref{a21}
 and Milstein-type scheme  \eqref{a211} recover
 the optimal convergence rate of order $0.5$ and $1$, respectively.
\end{Rem}

\section{The existence, uniqueness and H\"older continuity of the exact solution}\label{section3}
In this section we provide proofs for Theorems \ref{Theorem2.1} and
\ref{Theorem2.2}.
\begin{proof}[Proof of Theorem \ref{Theorem2.1}]    We assume $p\geq 2$. The case $1\leq p <2$ can be derived from
 Lyapunov inequality (namely $\|F\|_p\le \|F\|_q$ for
  $1\le p\le q\le \infty$).
 We borrow some ideas from \citep[Theorem 1]{Son2018Caputo}, where the authors studied the existence and uniqueness of the equations
$$
X(t)=X_{0}+\frac{1}{\Gamma(\alpha)}\bigg{(}\int_{0}^{t}(t-s)^{\alpha-1}b(X(s))ds+\int_{0}^{t}(t-s)^{\alpha-1}\sigma(X(s))dW_{s}\bigg{)},\quad
 \alpha>\frac{1}{2}
$$
in the space $L^2$.   Here, we consider SVIEs with two different singular kernel, which allow the singular parameter $\alpha$
in the drift term   vary from  $0$ to $1$ and
we consider the solution in $L^p$ for any $p\in [1, \infty)$. Moreover,  we shall prove that  $\TT: \mathbb{S}^{p}(0,1)\rightarrow \mathbb{S}^{p}(0,1)$ is a contraction mapping with respect to more general norm ($L^{p}$ norm, $p\geq 1$). Thus,    some new techniques will be needed.
Denote by $\mathbb{S}^{p}(0,1)$ the Banach
space of the stochastic process
 that are measurable, $\mathscr{F}_{t}$-adapted,  where   the norm
 of the process is defined  by
$$
\|X\|_{\mathbb{S}^{p}}:=\sup\limits_{0\leq t\leq 1}
\left(\EE  |X(t)|^p\right)^{1/p}  <\infty.
$$
Define operators $\TT: \mathbb{S}^{p}(0,1)\rightarrow \mathbb{S}^{p}(0,1)$ by
$$
\TT\lambda(t):=X_{0}+\int_{0}^{t}(t-s)^{-\alpha}a(\lambda(s))ds+\int_{0}^{t}(t-s)^{-\beta}b(\lambda(s))dW_{s}.
$$
 Obviously, the operators $\TT$ are well defined. Let $\kappa$ be a positive constant such that
\begin{align}\label{2.5}
\kappa>2^{p-1}\hat{L}^{p}\left(\frac{T^{(1-\alpha)p}}{1-\alpha}+\frac{T^{(1-2\beta)p}}{1-2\beta}\right)\Gamma(1-\max(\alpha,2\beta) ),
\end{align}%and a new weighted norm $\|\cdot\|_{\kappa}$ by
 where $\Gamma(\cdot)$ is the Gamma function.  We
 introduce  a new weighted norm $\|\cdot\|_{\kappa}$ by
$$
\|X\|_{\kappa}:=\sup_{t\in[0,1]}\sqrt{\frac{\mathbb{E}( |X(t) |^{p})}{E_{1-\max(\alpha,2\beta)}(\kappa t^{1-\max(\alpha,2\beta)})}},
$$
 where $E_{1-\max(\alpha,2\beta)}(\cdot)$ is the Mittag-Leffler function.
 It is easy to verify that $\|\cdot\|_{\mathbb{S}^{p}}$ and $\|\cdot\|_{\kappa}$ are equivalent.
 Next, we   show that $\TT$ is contractive with respect to the norm $\|\cdot\|_{\kappa}$. In fact,    for any  $ \lambda, \mu\in \mathbb{S}^{2}(0,1)$, we have by Jensen's inequality
\begin{eqnarray}
\mathbb{E} |\TT\lambda(t)-\TT\mu(t) |^{p}
&\leq& 2^{p-1}\mathbb{E}\left|\int_{0}^{t}(t-s)^{-\alpha}(a(\lambda(s))-a(\mu(s)))ds \right|^{p}\nonumber \\
&&+2^{p-1}\mathbb{E} \left|\int_{0}^{t}(t-s)^{-\beta}(b(\lambda(s))-b(\mu(s)))dW_{s} \right|^{p}\nonumber\\
&\leq& 2^{p-1}\left(\int_{0}^{t}(t-s)^{-\alpha}ds\right)^{p-1}\cdot \hat{L}^{p}
\int_{0}^{t}(t-s)^{-\alpha}\mathbb{E} |\lambda(s)-\mu(s) |^{p}ds\nonumber \\
&&+2^{p-1}\hat{L}^{p}\left(\int_{0}^{t}(t-s)^{-2\beta}ds\right)^{p/2-1}\mathbb{E}\int_{0}^{t}(t-s)^{-2\beta}
|\lambda(s)-\mu(s) |^{p}ds\nonumber\\
&\leq& 2^{p-1}\hat{L}^{p}\left(\frac{T^{(1-\alpha)p}}{1-\alpha}+\frac{T^{(1-2\beta)p}}{1-2\beta}\right)\int_{0}^{t}(t-s)^{-\max(\alpha,2\beta)}\mathbb{E} |\lambda(s)-\mu(s) |^{p}ds.\label{e.2.3}
\end{eqnarray}
Hence,  by Lemma \ref{Lemma 3.4} we have
\begin{align*}
&\frac{\mathbb{E} |\TT\lambda(t)-\TT\mu(t) |^{p}}{E_{1-\max(\alpha,2\beta)}(\kappa t^{1-\max(\alpha,2\beta)})}\\
&\leq 2^{p-1}\hat{L}^{p}\left(\frac{T^{(1-\alpha)p}}{1-\alpha}+\frac{T^{(1-2\beta)p}}{1-2\beta}\right)
\frac{\int_{0}^{t}(t-s)^{-\max(\alpha,2\beta)}E_{1-\max(\alpha,2\beta)}(\kappa s^{1-\max(\alpha,2\beta)})ds}
{E_{1-\max(\alpha,2\beta)}(\kappa t^{1-\max(\alpha,2\beta)})}\|\lambda-\mu\|_{\kappa}^{p}\\
&\leq 2^{p-1}\hat{L}^{p}\left(\frac{T^{(1-\alpha)p}}{1-\alpha}+\frac{T^{(1-2\beta)p}}{1-2\beta}\right)
\frac{\Gamma(1-\max(\alpha,2\beta))}{\kappa}\|\lambda-\mu\|_{\kappa}^{p}\\
&=  \rho\|\lambda-\mu\|_{\kappa}^{p},
\end{align*}
where
\[
\rho=2^{p-1}\hat{L}^{p}\left(\frac{T^{(1-\alpha)p}}{1-\alpha}+\frac{T^{(1-2\beta)p}}{1-2\beta}\right)
\frac{\Gamma(1-\max(\alpha,2\beta))}{\kappa}
\]
and where we used $$
\int_{0}^{t}(t-s)^{-\max(\alpha,2\beta)}E_{1-\max(\alpha,2\beta)}(\kappa s^{1-\max(\alpha,2\beta)})ds\leq \frac{\Gamma(1-\max(\alpha,2\beta))}{\kappa}
E_{1-\max(\alpha,2\beta)}(\kappa t^{1-\max(\alpha,2\beta)})\,.
$$
%Therefore
%$$
%\|\TT\lambda(t)-\TT\mu(t)\|^{2}_{\kappa}\leq \rho\|\lambda-\mu\|_{\kappa}^{2}.
%$$
By our choice of $\kappa$ (namely \eqref{2.5})
we see that  that $\rho <1$.  Thus we   conclude that
$\TT: \mathbb{S}^{p}(0,1)\rightarrow \mathbb{S}^{p}(0,1)$ is a contraction mapping. By Banach contractive mapping theorem,
 we   see  that there exists a unique solution in $\mathbb{S}^{p}(0,1)$.

The bound \eqref{e.2.1} follows easily from the above argument
(e.g. \eqref{e.2.3} with $\mu=0$)
and the linear growth condition of the functions
  $a$ and $b$.
\end{proof}

\begin{proof}[Proof of Theorem \ref{Theorem2.2}]
 We continue to  assume $p\ge 2$.
The case $1\leq p<2$ can be proved by using a  Lyapunov inequality. \
%{\color{red}The remaining parts of the paper
%  use the same technique, which will no longer be emphasized.}
 Let $X(t)$ satisfy \eqref{a1}.  We can write
\begin{align*}
X(t)-X(r)=&\bigg{(}\int_{0}^{t}(t-s)^{-\alpha}a(X(s))ds-\int_{0}^{r}(r-s)^{-\alpha}a(X(s))ds\bigg{)}\\
&+\bigg{(}\int_{0}^{t}(t-s)^{-\beta}b(X(s))dW_{s}-\int_{0}^{r }(r -s)^{-\beta}b(X(s))dW_{s}\bigg{)}\\
&=:I_{41}+I_{42}.
\end{align*}
Obviously, $I_{42}$ can be written as
$$
I_{42}=\int_{0}^{r }[(t-s)^{-\beta}-(r -s)^{-\beta}]b(X(s))dW_{s}+\int_{r }^{t}(t-s)^{-\beta}b(X(s))dW_{s}=:I_{421}+I_{422}\,.
$$
Then by Burkholder-Davis-Gundy inequality, we have
\begin{align*}
\mathbb{E}|I_{421}|^p =&   \EE\left|
\int_{0}^{r }[(t-s)^{-\beta}-(r -s)^{-\beta}]b(X(s))dW_{s}
\right|^p  \\
\leq& C_p   \mathbb{E}
\left( \int_{0}^{r }\left| [(t-s)^{-\beta}-(r -s)^{-\beta}]b(X(s))\right|^2 ds  \right)^{p/2}  \,.
\end{align*}
Denote
\[
\rho_{t,r}:=\int_{0}^{r }[(t-s)^{-\beta}-(r -s)^{-\beta}]^{2}ds\,.
\]
Since $\phi(x)=x^{p/2}$, $x>0$ is convex,  applying Jensen's
 inequality  we have
\begin{eqnarray}
&& \left( \frac{1}{\rho_{t,r}}\int_{0}^{r }\left| [(t-s)^{-\beta}-(r -s)^{-\beta}]b(X(s))\right|^2 ds  \right)^{p/2}\nonumber\\
&&\qquad  \le   \frac{1}{\rho_{t,r}}\int_{0}^{r }\left| (t-s)^{-\beta}-(r -s)^{-\beta}\right|^2\left|b(X(s))\right|^p ds\,.
 \end{eqnarray}
Thus we have
\begin{align}
\mathbb{E}|I_{421}|^p
\leq&   C_p \rho_{t,r}^{\frac p2 -1}  \int_{0}^{r }\left| (t-s)^{-\beta}-(r -s)^{-\beta}\right|^2\EE \left|b(X(s))\right|^p ds\nonumber\\
\leq&  C_p  \rho_{t,r}^{\frac p2 -1}  \int_{0}^{r }\left| (t-s)^{-\beta}-(r -s)^{-\beta}\right|^2\EE (1+\left| X(s) \right|)^p ds\nonumber\\
\leq&  C_p  \rho_{t,r}^{\frac p2 -1}  \int_{0}^{r }\left| (t-s)^{-\beta}-(r -s)^{-\beta}\right|^2  ds = \rho_{t,r}^{\frac p2  } \,.
\label{e.2.6}
\end{align}
Now we need to obtain a sharp bound on $\rho_{t,r}$
\begin{align*}
\rho_{t,r} =&  \int_{0}^{r }[(t-s)^{-\beta}-(r -s)^{-\beta}]^{2}ds \\
=& \beta^2 \int_{0}^{r }\left(\int_r^t (\tau-s)^{-\beta-1} d\tau \right)^2ds\\
=& \beta^2 \int_{0}^{r } \int_r^t\int_r^t
 (\tau_1-s)^{-\beta-1} (\tau_2-s)^{-\beta-1} d\tau_1 d\tau_2  ds\\
 =& 2\beta^2  \int_{r<\tau_1<\tau_2\le t}  \int_0^r
 (\tau_1-s)^{-\beta-1} (\tau_2-s)^{-\beta-1} ds d\tau_1 d\tau_2  \\
 \leq& 2\beta^2  \int_{r<\tau_1<\tau_2\le t}  \int_0^r
 (\tau_1-s)^{-\beta-1} (\tau_2-r)^{-\beta-1} ds d\tau_1 d\tau_2  \\
 =& 2\beta   \int_{r<\tau_1<\tau_2\le t}
 \left[ (\tau_1-r)^{-\beta }-\tau_1^{-\beta}\right]
  (\tau_2-r)^{-\beta-1}   d\tau_1 d\tau_2  \\
\leq & 2\beta   \int_{r<\tau_1<\tau_2\le t}
   (\tau_1-r)^{-\beta }
  (\tau_2-r)^{-\beta-1}   d\tau_1 d\tau_2  \\
= & \frac{2\beta }{ 1-\beta }  \int_{r}^{t}
  (\tau_2-r)^{-2\beta }     d\tau_2  \\
  = & \frac{2\beta }{ (1-\beta)(1-2\beta) }
  (t-r)^{1-2\beta }      \,.
\end{align*}
This together with \eqref{e.2.6} implies
\begin{equation}
\EE |I_{421}|^p\le C_p (t-r)^{p(\frac12-\beta)}\,. \label{e.2.7}
\end{equation}
Analogously  to \eqref{e.2.6}, if we denote
\[
\tilde \rho_{t,r}=\int_r^t(t-s)^{-2\beta} ds=\frac{(t-r)^{1-2\beta}}{
1-2\beta}\,.
\]
Then we have
\begin{equation}
\EE |I_{422}|^p\le C_p \tilde \rho_{t,r}^{p/2}\le
C_p (t-r)^{p(\frac12-\beta)}\,. \label{e.2.8}
\end{equation}
Combining this with \eqref{e.2.7} we have
\begin{equation}
\EE |I_{42}|^p\le C_p (t-r)^{p(\frac12-\beta)}\,.
\label{e.2.9}
\end{equation}
%\begin{align*}
%\mathbb{E}|I_{42}|^{2}\le& 2^p\left(\EE\left|
%\int_{0}^{r }[(t-s)^{-\beta}-(r -s)^{-\beta}]b(X(s))dW_{s}
%\right|^p +\EE\left|\int_{r }^{t}(t-s)^{-\beta}b(X(s))dW_{s}
%\right|^p\right) \\
%\leq& C_p \left( \mathbb{E}
%\left( \int_{0}^{r }\left| [(t-s)^{-\beta}-(r -s)^{-\beta}]b(X(s))\right|^2 ds  \right)^{p/2} +
% \mathbb{E}\left(\bigg{|}\int_{r }^{t}(t-s)^{-\beta}b(X(s))dW_{s}\bigg{|}^{2}\right)^{p/2}\right) \\
%\leq& C\int_{0}^{r }[(t-s)^{-\beta}-(r -s)^{-\beta}]^{2}ds+C\int_{r }^{t}(t-s)^{-2\beta}ds\\
%\leq& C\sum_{k=0}^{i-2}\int_{t_{k}}^{t_{k+1}}|\beta\int_{r }^{t}(\tau-s)^{-\beta-1}d\tau|^{2}ds+C\int_{t_{i-1}}^{r }
%|(t-s)^{-\beta}-(r -s)^{-\beta}|^{2}ds+Ch^{1-2\beta}\\
%\leq& C\beta^{2}\sum_{k=0}^{i-2}\int_{t_{k}}^{t_{k+1}}\bigg{(}\int_{r }^{t}(r -t_{k+1})^{-\beta-1}d\tau\bigg{)}^{2}ds+Ch^{1-2\beta}\\
%\leq& Ch\sum_{k=0}^{i-2}(i-k-1)^{-2(\beta+1)}h^{-2\beta}+Ch^{1-2\beta}\leq Ch^{1-2\beta},
%\end{align*}
%where Assumption \ref{assu1} and Lemma \ref{lemma_2.3} have been used.
For the term $I_{41}$
 \begin{align*}
 \mathbb{E}|I_{41}|^p\leq& 2^p \mathbb{E}\left|\int_{0}^{r }[(t-s)^{-\alpha}-(r -s)^{-\alpha}]a(X(s))ds\right|^{p}\\
 &\qquad
 +2^p\mathbb{E}\left|\int_{r }^{t}(t-s)^{-\alpha}a(X(s))ds\right|^{p}
 =:I_{411}+I_{412}\,.
 \end{align*}
In the same way as for \eqref{e.2.7} we have
  \begin{align*}
 I_{411} \leq&  \left(\int_0^r \left|(t-s)^{-\alpha}-(r -s)^{-\alpha}
 \right| ds\right)^{p}\\
 \le&  C_p \left(\int_0^r  \int_r^t (\tau -s)^{-\al-1} d\tau ds\right)^{p}\\
=& C_p \left(   \int_r^t \int_0^r (\tau -s)^{-\al-1} ds d\tau  \right)^{p}
 \\
\le & C_p \left(   \int_r^t  (\tau -r)^{-\al }   d\tau  \right)^{p}
=C_p(t-r)^{(1-\alpha)p}\,.  \\
 \end{align*}
In the similar way, we can prove that
$$
\mathbb{E}|I_{412}|^{2}\leq C_p (t-r)^{p(1-\alpha)}.
$$
This completes the proof of the theorem.
\end{proof}

%The next stability lemma plays an important role in the convergence analysis.
%\begin{Lem}
%If $(\Psi, h, \xi)$ is stochastically C-stable one-step method with constants $C$, and $\eta \in (1,\infty)$,  then for every grid function $Z\in \mathscr{G}^{2}(\mathcal{T}_{h})$,
%\begin{align}\label{a7}
%\begin{split}
%\max_{n\in\{0\cdots N\}}&\|Z(t_{n})-X_{h}(t_{n})\|_{L^{2}(\Omega;\mathbb{R}^{d})}^{2}\\
%\leq& e^{2(1+C(1+h))T}\bigg{(}\sum_{i=1}^{M}\big{\|}Z(t_{i-M})-\xi(t_{i-M})\big{\|}_{L^{2}(\Omega;\mathbb{R}^{d})}^{2}+
%\|Z(t_{0})-X_{h}(t_{0})\|_{L^{2}(\Omega;\mathbb{R}^{d})}^{2}\\
%&+\sum_{i=1}^{N}(1+h^{-1})\big{\|}\mathbb{E}\big[Z(r )-\Psi(Z(t_{i-1}),Z(t_{i-M}),h)|\mathscr{F}_{t_{i-1}}]\big{\|}_{L^{2}(\Omega;\mathbb{R}^{d})}^{2}\\
%&+C_{\eta}\sum_{i=1}^{N}\big{\|}(id-\mathbb{E}[\cdot|\mathscr{F}_{t_{i-1}}])\big{(}Z(t_{i})-\Psi(Z(t_{i-1}),Z(t_{i-M}),h)\big{)}\big{\|}_{L^{2}(\Omega;\mathbb{R}^{d})}^{2}
% \bigg{)},
%\end{split}
%\end{align}
%where $Z(t_{i-M})$, $\xi(t_{i-M}), i=0, 1, \cdots, M$, are defined by $Z(t_{i}-\tau)$ and $\xi(t_{i}-\tau)$, respectively.
%\end{Lem}
%
%For the convergence analysis, we also need the following generalized discrete Gronwall inequality.

\section{Convergence rate of $\theta$-Euler-Maruyama
scheme }\label{section4}
In this section we provide proof of the $\theta$ Euler-Maruyama
scheme, namely, Theorem \ref{theorem2.3}.
We denote  the   local truncation errors  of the $\theta$ Euler-Maruyama
scheme by
\begin{align}\label{a22}
\begin{split}
R_{h}^{E}(t_{n+1})=&\theta\sum_{i=0}^{n}\int_{t_{i}}^{t_{i+1}}(t_{n+1}-s)^{-\alpha}(a(X(s))-a(X(t_{i+1})))ds\\
&+(1-\theta)\sum_{i=0}^{n}
\int_{t_{i}}^{t_{i+1}}(t_{n+1}-s)^{-\alpha}(a(X(s))-a(X(t_{i})))ds\\
&+\sum_{i=0}^{n}\int_{t_{i}}^{t_{i+1}}(t_{n+1}-s)^{-\beta}b(X(s))-(t_{n+1}-t_{i})^{-\beta}b(X(t_{i}))dW_{s}\\
=&:I_{1}+I_{2}+I_{3}\,.
\end{split}
\end{align}

\begin{Lem}\label{Lemma 2.1}
If Assumption \ref{assu1} holds, then for the local truncation error $R_{h}^{E}(t_{n+1})$, there is a constant $C$ such that
 for any $p\geq 1$
$$
\mathbb{E}|R_{h}^{E}(t_{n+1})|^{p}\leq Ch^{\gamma(\alpha, \beta)p},
$$
 where $\gamma(\alpha, \beta)=\min\{\frac12 - \beta,~  1-\alpha\}$.
\end{Lem}
\begin{proof}
 By \eqref{a22}
\begin{align}
\begin{split}
\mathbb{E}|I_{3}|^{p}\leq& \mathbb{E}\left|\int_{0}^{t_{n+1}}[(t_{n+1}-s)^{-\beta}-(t_{n+1}-\eta(s))^{-\beta}]b(X(s))dW_{s}\right|^{p}\\
&+\mathbb{E}\left|\int_{0}^{t_{n+1}}(t_{n+1}-\eta(s))^{-\beta}[b(X(s))-b(X({\eta(s)}))]dW_{s}\right|^{p}=:I_{31}+I_{32}.
\end{split}
\end{align}
Furthermore, using Assumption \ref{assu1} and Theorem \ref{Theorem2.1}, we have
\begin{align*}
I_{31}\leq& \mathbb{E}\bigg{(}\int_{0}^{t_{n+1}}[(t_{n+1}-s)^{-\beta}-(t_{n+1}-\eta(s))^{-\beta}]^{2}|b(X(s))|^{2}ds\bigg{)}^{\frac{p}{2}}.
\end{align*}
Let
$$
\rho_{n+1}=\int_{0}^{t_{n+1}}[(t_{n+1}-s)^{-\beta}-(t_{n+1}-\eta(s))^{-\beta}]^{2}ds\,.
$$
Since $\phi(x)=x^{p/2}\,, x>0$   is convex by Jensen's inequality,  we have
\begin{align*}
&\left(\frac{1}{\rho_{n+1}}\int_{0}^{t_{n+1}}[(t_{n+1}-s)^{-\beta}-(t_{n+1}-\eta(s))^{-\beta}]^{2}|b(X(s))|^{2}ds\right)^{\frac{p}{2}}\\
&\leq \frac{1}{\rho_{n+1}}\int_{0}^{t_{n+1}}[(t_{n+1}-s)^{-\beta}-(t_{n+1}-\eta(s))^{-\beta}]^{2}|b(X(s))|^{p}ds.
\end{align*}
Hence,
\begin{align}\label{q1}
I_{31}\leq C\rho_{n+1}^{\frac{p}{2}-1}\int_{0}^{t_{n+1}}[(t_{n+1}-s)^{-\beta}-(t_{n+1}-\eta(s))^{-\beta}]^{2}\mathbb{E}(1+|X(s)|)^{p}ds
\leq C\rho_{n+1}^{\frac{p}{2}}.
\end{align}
Now, we give a sharp estimate for $\rho_{n+1}$
\begin{align}\label{q2}
\nonumber\rho_{n+1}=&C \int_{0}^{t_{n+1}}[(t_{n+1}-s)^{-\beta}-(t_{n+1}-\eta(s))^{-\beta}]^{2}ds\\
\nonumber=&C\sum_{i=0}^{n}\int_{t_{i}}^{t_{i+1}}\left|\beta\int_{t_{i}}^{s}(t_{n+1}-\tau)^{-\beta-1}d\tau\right|^{2}ds\\
\leq& C\beta^{2}\sum_{i=0}^{n}\int_{t_{i}}^{t_{i+1}}\left|\int_{t_{i}}^{t_{i+1}}(t_{n+1}-t_{i+1})^{-\beta-1}d\tau\right|^{2}ds\\
\nonumber\leq& C\beta^{2}h^{1-2\beta}\sum_{i=0}^{n}(n-i)^{-2(\beta+1)}\leq Ch^{1-2\beta}.
\end{align}
Combining this  with \eqref{q1}, we have
$$
I_{31}\leq Ch^{(1-2\beta)\frac{p}{2}}.
$$
Let $\widetilde{\rho}_{n+1}=\int_{0}^{t_{n+1}}(t_{n+1}-\eta(s))^{-2\beta}ds$. By  an  analysis similar to the
 above, we have
$$
I_{32}\leq C_{p}\widetilde{\rho}_{n+1}^{\frac{p}{2}-1}\int_{0}^{t_{n+1}}(t_{n+1}-\eta(s))^{2}\mathbb{E}|X(s)-X(\eta(s))|^{p}ds,
$$
 where Assumption \ref{assu1} was  used. Using Theorem \ref{Theorem2.2}, one sees that
$$
I_{32}\leq Ch^{\gamma(\alpha, \beta)p}.
$$
Thus,
$$
\mathbb{E}|I_{3}|^{2}\leq Ch^{\gamma(\alpha, \beta)p}.
$$
 In a similar manner, by H\"{o}lder inequality and Jensen's inequality, one also has
$$
\mathbb{E}|I_{1}|^{2}\leq Ch^{p(1-\alpha)},~~~\mathbb{E}|I_{2}|^{2}\leq Ch^{p(1-\alpha)}.
$$
Summarizing the above arguments the desired assertion follows.
\end{proof}

%The following theorem shows that convergence can be derived from stability plus consistency.

\begin{proof}[Proof of Theorem \ref{theorem2.3}]
Denote  $\varepsilon_{n+1}:=X(t_{n+1})-Y_{n+1}$. It follows from \eqref{a1}, \eqref{a21} and \eqref{a22}
\begin{align}\label{aa1}
\begin{split}
\varepsilon_{n+1}=&R_{h}^{E}(t_{n+1})+\theta\sum_{i=0}^{n}\int_{t_{i}}^{t_{i+1}}(t_{n+1}-s)^{-\alpha}
(a(X_{t_{i+1}})-a(Y_{i+1}))ds\\
&+(1-\theta)\sum_{i=0}^{n}\int_{t_{i}}^{t_{i+1}}(t_{n+1}-s)^{-\alpha}(a(X_{t_{i}})-a(Y_{i}))ds\\
&+\sum_{i=0}^{n}\int_{t_{i}}^{t_{i+1}}(t_{n+1}-s)^{-\beta}(b(X_{t_{i}})-b(Y_{i}))dW_{s}.
\end{split}
\end{align}
 Let
 $$
 \varepsilon_{\hat{s}}=\sum_{i=0}^{N-1}\varepsilon_{i+1}~\chi_{s\in [t_{i}, t_{i+1}]},~~~
 \varepsilon_{\check{s}}=\sum_{i=0}^{N-1}\varepsilon_{i}~\chi_{s\in [t_{i}, t_{i+1}]},
 $$
 where
$$
\chi_{s\in [t_{i}, t_{i+1}]}=\left\{
                               \begin{array}{ll}
                                 1, & \hbox{$s\in [t_{i}, t_{i+1}]$;} \\
                                 0, & \hbox{\text{otherwise}.}
                               \end{array}
                             \right.
$$
  Then
\begin{eqnarray}\label{aa2}
%\begin{split}
\mathbb{E}|\varepsilon_{n+1}|^{p}&\leq& 4^{p-1}\mathbb{E}|R_{h}^{E}(t_{n+1})|^{p}+4^{p-1}\theta^{p}\mathbb{E}
\big{|}\int_{0}^{t_{n+1}}(t_{n+1}-s)^{-\alpha}(a(X_{t_{i+1}})-a(Y_{i+1}))ds\big{|}^{p}\nonumber\\
&&+4^{p-1}(1-\theta)^{p}\mathbb{E}\big{|}\int_{0}^{t_{n+1}}(t_{n+1}-s)^{-\alpha}(a(X_{t_{i}})-a(Y_{i}))ds\big{|}^{p}\nonumber\\
&&+4^{p-1}\mathbb{E}\big{|}\int_{0}^{t_{n+1}}
(t_{n+1}-\eta(s))^{-\beta}(b(X_{t_{i}})-b(Y_{i}))dW_{s}\big{|}^{p}\nonumber\\
&\leq& C\bigg{\{}\mathbb{E}|R_{h}^{E}(t_{n+1})|^{p}
+\bigg{(}\int_{0}^{t_{n+1}}(t_{n+1}-s)^{-\alpha}ds\bigg{)}^{p-1}
\int_{0}^{t_{n+1}}(t_{n+1}-s)^{-\alpha}\mathbb{E}|\varepsilon_{\hat{s}}|^{p}ds\nonumber\\
&&+\bigg{(}\int_{0}^{t_{n+1}}(t_{n+1}-s)^{-\alpha}ds\bigg{)}^{p-1}
\int_{0}^{t_{n+1}}(t_{n+1}-s)^{-\alpha}\mathbb{E}|\varepsilon_{\check{s}}|^{p}ds\nonumber\\
&&+
\left(\int_{0}^{t_{n+1}}(t_{n+1}-\eta(s))^{-2\beta}ds\right)^{\frac{p}{2}-1}\int_{0}^{t_{n+1}}
(t_{n+1}-s)^{-2\beta}\mathbb{E}|\varepsilon_{\check{s}}|^{p}ds\bigg{\}},
%\end{split}
\end{eqnarray}
 where Assumption \ref{assu1} and Jensen's inequality were  used.
% It is easy to verify that
%$$
%\int_{0}^{t_{n+1}}(t_{n+1}-\eta(s))^{-2\beta}\mathbb{E}|\varepsilon_{\check{s}}|^{2}ds=\sum_{i=0}^{n}
%\int_{t_{i}}^{t_{i+1}}(t_{n+1}-\eta(s))^{-2\beta}ds
%\mathbb{E}|\varepsilon_{i}|^{2}.
%$$
Then,
\begin{align}\label{aa3}
\int_{t_{i}}^{t_{i+1}}(t_{n+1}-s)^{-\alpha}ds=\frac{(t_{n+1}-t_{i})^{1-\alpha}-(t_{n+1}-t_{i+1})^{1-\alpha}}{1-\alpha}
=\frac{h^{1-\alpha}[(n+1-i)^{1-\alpha}-(n-i)^{1-\alpha}]}{1-\alpha}.
\end{align}
% for $i=n$, it is obvious that
%$$
%\int_{t_{i}}^{t_{i+1}}(t_{n+1}-\eta(s))^{-2\beta}ds=\int_{t_{n}}^{t_{n+1}}(t_{n+1}-t_{n})^{-2\beta}ds= h^{1-2\beta}.
%$$
%For $i\leq n-1$, we can find that
%\begin{align*}
%\int_{t_{i}}^{t_{i+1}}(t_{n+1}-\eta(s))^{-2\beta}ds=\int_{t_{i}}^{t_{i+1}}(t_{n+1}-t_{i})^{-2\beta}ds=h^{1-2\beta}(n+1-i)^{-2\beta}.
%\end{align*}
Note that
\begin{align*}
&(n+1-i)^{1-\alpha}-(n-i)^{1-\alpha}\\
&=(n+1-i)^{1-\alpha}[1-(1-\frac{1}{n+1-i})^{1-\alpha}]\\
&=(n+1-i)^{-\alpha}\big{[}1-\theta_{n,i}\frac{1}{n+1-i}\big{]}^{-\alpha}.
\end{align*}
Since $1-\theta_{n,i}\frac{1}{n+1-i}\geq \frac{1}{2}$, we have
\begin{align}\label{aa4}
(n+1-i)^{1-\alpha}-(n-i)^{1-\alpha}\leq 2^{\alpha}(n+1-i)^{-\alpha}.
\end{align}
 Combining the above results with  \eqref{aa2},%, \eqref{aa3} and \eqref{aa4},
 we obtain
\begin{align}\label{aa5}
\begin{split}
\mathbb{E}|\varepsilon_{n+1}|^{p}\leq& C\bigg{[}\mathbb{E}|R_{h}^{E}(t_{n+1})|^{p}+h^{1-\alpha}\sum_{i=0}^{n}\mathbb{E}|\varepsilon_{i+1}|^{p}\\
&+
h^{1-\max(\alpha,2\beta)}\sum_{i=0}^{n}\mathbb{E}|\varepsilon_{i}|^{p}+h^{1-\max(\alpha,2\beta)}\sum_{i=0}^{n}
(n+1-i)^{-\max(\alpha,2\beta)}\mathbb{E}|\varepsilon_{i}|^{p}\bigg{]}.
\end{split}
\end{align}
 The final result follows from Lemma \ref{Lemma 2.1} and Lemma \ref{Lemma 3.2}.
\end{proof}

\section{Convergence rate of Milstein-type scheme}\label{section5}
In this section we provide a proof for the strong convergence rate of Milstein scheme, i.e.  Theorem \ref{t.2.4}. We denote
\begin{align}\label{a222}
\begin{split}
R_{h}^{M}(t_{n+1})=&X(t_{n+1})-X_{0}-\sum_{i=0}^{n}\int_{t_{i}}^{t_{i+1}}(t_{n+1}-s)^{-\alpha}a(X(t_{i}))ds-
\sum_{i=0}^{n}\int_{t_{i}}^{t_{i+1}}(t_{n+1}-s)^{-\beta}b(X(t_{i}))dW_{s}\\
&-\sum_{i=0}^{n}\int_{t_{i}}^{t_{i+1}}(t_{n+1}-s)^{-\beta}b^{'}(X(t_{i}))\bigg{(}\int_{0}^{t_{i}}[(s-r)^{-\beta}
-(t_{i}-r)^{-\beta}]b(X(r))dW_{r}\bigg{)}dW_{s}\\
&-\sum_{i=0}^{n}\int_{t_{i}}^{t_{i+1}}(t_{n+1}-s)^{-\beta}b^{'}(X(t_{i}))\bigg{(}\int_{t_{i}}^{s}(s-r)^{-\beta}b(X(r))dW_{r}\bigg{)}dW_{s}\,.
\end{split}
\end{align}

%From Theorem \ref{theorem2.3}, we know that the Euler method has a low order of convergence. A natural question which arises is: whether there exist any higher order methods, for example, the Milstein-type method.
% We will answer this question in this section.

From \eqref{a1}, we have
\begin{align}\label{aa6}
X({t_{n+1}})=X_{0}+\sum_{i=0}^{n}\int_{t_{i}}^{t_{i+1}}(t_{n+1}-s)^{-\alpha}a(X(s))ds
+\sum_{i=0}^{n}\int_{t_{i}}^{t_{i+1}}(t_{n+1}-s)^{-\beta}b(X(s))dW_{s}.
\end{align}
Using the Taylor expansions of function $a(X(s))$ and $b(X(s))$ at $X(t_{i})$, we get
\begin{align}\label{aa7}
a(X(s))=&a(X(t_{i}))+a^{'}(X(t_{i}))(X(s)-X(t_{i}))\nonumber\\
&+\frac{a^{''}(X(t_{i}))}{2!}(X(s)-X(t_{i}))^{2}
+\frac{a^{'''}(X_{t_{i,\theta_{1}}})}{3!}(X(s)-X(t_{i}))^{3},
\end{align}
and
\begin{align}\label{aa8}
b(X(s))=&b(X(t_{i}))+b^{'}(X(t_{i}))(X(s)-X(t_{i}))\nonumber\\
&+\frac{b^{''}(X(t_{i}))}{2!}(X(s)-X(t_{i}))^{2}
+\frac{b^{'''}(X_{t_{i,\theta_{2}}})}{3!}(X(s)-X(t_{i}))^{3},
\end{align}
 where $X_{t_{i,\theta_{1}}}$, $X_{t_{i,\theta_{2}}}$ are between  $X({t_{i}})$ and $X(s)$.
Substituting \eqref{aa7} and \eqref{aa8} to \eqref{aa6}, one finds
\begin{align}\label{aa9}
\nonumber X({t_{n+1}})=&X_{0}+\sum_{i=0}^{n}\int_{t_{i}}^{t_{i+1}}(t_{n+1}-s)^{-\alpha}a(X(t_{i}))ds
+\sum_{i=0}^{n}\int_{t_{i}}^{t_{i+1}}(t_{n+1}-s)^{-\beta}b(X(t_{i}))dW_{s}\\
\nonumber&+\sum_{i=0}^{n}\int_{t_{i}}^{t_{i+1}}(t_{n+1}-s)^{-\alpha}a^{'}(X(t_{i}))(X(s)-X(t_{i}))ds\\
\nonumber&+\sum_{i=0}^{n}\int_{t_{i}}^{t_{i+1}}(t_{n+1}-s)^{-\beta}b^{'}(X(t_{i}))(X(s)-X(t_{i}))dW_{s}\\
&+\frac{1}{2!}\sum_{i=0}^{n}\int_{t_{i}}^{t_{i+1}}(t_{n+1}-s)^{-\alpha}a^{''}(X(t_{i}))(X(s)-X(t_{i}))^{2}ds\\
\nonumber&+\frac{1}{2!}\sum_{i=0}^{n}\int_{t_{i}}^{t_{i+1}}(t_{n+1}-s)^{-\beta}b^{''}(X(t_{i}))(X(s)-X(t_{i}))^{2}dW_{s}\\
\nonumber&+\frac{1}{3!}\sum_{i=0}^{n}\int_{t_{i}}^{t_{i+1}}(t_{n+1}-s)^{-\alpha}a^{'''}(X_{t_{i,\theta_{1}}})(X(s)-X(t_{i}))^{3}ds\\
\nonumber&+\frac{1}{3!}\sum_{i=0}^{n}\int_{t_{i}}^{t_{i+1}}(t_{n+1}-s)^{-\beta}b^{'''}(X_{t_{i,\theta_{2}}})(X(s)-X(t_{i}))^{3}dW_{s}.
\end{align}
It follows from \eqref{a1} that,
\begin{align}\label{aa10}
\begin{split}
X(s)-X(t_{i})=&\int_{0}^{t_{i}}[(s-r)^{-\alpha}-(t_{i}-r)^{-\alpha}]a(X(r))dr+\int_{t_{i}}^{s}(s-r)^{-\alpha}a(X(r))dr\\
&+\int_{0}^{t_{i}}[(s-r)^{-\beta}-(t_{i}-r)^{-\beta}]b(X(r))dW_{r}+\int_{t_{i}}^{s}(s-r)^{-\beta}b(X(r))dW_{r}.
\end{split}
\end{align}
Substituting \eqref{aa10} into \eqref{aa9}, one arrives at
\begin{align}\label{aa11}
\begin{split}
X({t_{n+1}})=&X_{0}+\sum_{i=0}^{n}\int_{t_{i}}^{t_{i+1}}(t_{n+1}-s)^{-\alpha}a(X(t_{i}))ds
+\sum_{i=0}^{n}\int_{t_{i}}^{t_{i+1}}(t_{n+1}-s)^{-\beta}b(X(t_{i}))dW_{s}\\
&+\sum_{i=0}^{n}\int_{t_{i}}^{t_{i+1}}(t_{n+1}-s)^{-\beta}b^{'}(X(t_{i}))
\int_{0}^{t_{i}}[(s-r)^{-\beta}-(t_{i}-r)^{-\beta}]b(X(r))dW_{r}dW_{s}\\
&+\sum_{i=0}^{n}\int_{t_{i}}^{t_{i+1}}(t_{n+1}-s)^{-\beta}b^{'}(X(t_{i}))
\int_{t_{i}}^{s}(s-r)^{-\beta}b(X(r))dW_{r}dW_{s}\\
&+\tilde{\beta}+\tilde{\gamma}+\tilde{\delta}+R_{n},
\end{split}
\end{align}
where
%\begin{align*}
%\tilde{\alpha}=&\sum_{i=0}^{n}\int_{t_{i}}^{t_{i+1}}(t_{n+1}-s)^{-\alpha}b^{'}(X(t_{i}))
%\int_{0}^{t_{i}}[(s-r)^{-\alpha}-(t_{i}-r)^{-\alpha}]b(X(r))dW_{r}dW_{s}\\
%&+\sum_{i=0}^{n}\int_{t_{i}}^{t_{i+1}}(t_{n+1}-s)^{-\alpha}b^{'}(X(t_{i}))
%\int_{t_{i}}^{s}(s-r)^{-\alpha}b(X(r))dW_{r}dW_{s},
%\end{align*}

\begin{align*}
\tilde{\beta}=&\sum_{i=0}^{n}\int_{t_{i}}^{t_{i+1}}(t_{n+1}-s)^{-\beta}b^{'}(X(t_{i}))\int_{0}^{t_{i}}
[(s-r)^{-\alpha}-(t_{i}-r)^{-\alpha}]a(X(r))drdW_{s}\\
&+\sum_{i=0}^{n}\int_{t_{i}}^{t_{i+1}}(t_{n+1}-s)^{-\beta}b^{'}(X(t_{i}))
\int_{t_{i}}^{s}(s-r)^{-\alpha}a(X(r))drdW_{s},
\end{align*}

\begin{align*}
\tilde{\gamma}=&\sum_{i=0}^{n}\int_{t_{i}}^{t_{i+1}}(t_{n+1}-s)^{-\alpha}a^{'}(X(t_{i}))\int_{0}^{t_{i}}[(s-r)^{-\beta}-(t_{i}-r)^{-\beta}]
b(X(r))dW_{r}ds\\
&+\sum_{i=0}^{n}\int_{t_{i}}^{t_{i+1}}(t_{n+1}-s)^{-\alpha}a^{'}(X(t_{i}))\int_{t_{i}}^{s}(s-r)^{-\beta}b(X(r))dW_{r}ds,
\end{align*}
and
\begin{align*}
\tilde{\delta}=&\frac{1}{2}\sum_{i=0}^{n}\int_{t_{i}}^{t_{i+1}}(t_{n+1}-s)^{-\beta}b^{''}(X(t_{i}))
\bigg{[}\int_{0}^{t_{i}}[(s-r)^{-\beta}-(t_{i}-r)^{-\beta}]b(X(r))dW_{r}\\
&+\int_{t_{i}}^{s}(s-r)^{-\beta}b(X(r))dW_{r}\bigg{]}^{2}dW_{s}.
\end{align*}

Under Assumption \ref{assu1}, we now prove   that the numerical solution has a bounded moment of order $p ~(p\geq 1)$.
\begin{Lem}\label{lemma 4.1}
Assume the derivative of $b$ is bounded and Assumption \ref{assu1} holds. Then there is a constant $C$ such that for any $p\geq 1$
$$
\max\limits_{0\leq i\leq N}\mathbb{E}|Z_{i}|^{p}\leq C.
$$
\end{Lem}
\begin{proof}  We can rewrite the equation \eqref{a211} in a continuous form as follows
\begin{align}\label{aa14}
\begin{split}
Z_{n+1}=&Z_{0}+\int_{0}^{t_{n+1}}(t_{n+1}-s)^{-\alpha}a(Z_{\lfloor s/h\rfloor})ds+\int_{0}^{t_{n+1}}(t_{n+1}-s)^{-\beta}b(Z_{\lfloor s/h\rfloor})dW_{s}\\
&+\int_{0}^{t_{n+1}}(t_{n+1}-s)^{-\beta}b^{'}(Z_{\lfloor s/h\rfloor})\int_{0}^{\eta(s)}[(s-r)^{-\beta}-(\eta(s)-r)^{-\beta}]b(Z_{\lfloor r/h\rfloor})dW_{r}dW_{s}\\
&+\int_{0}^{t_{n+1}}(t_{n+1}-s)^{-\beta}b^{'}(Z_{\lfloor s/h\rfloor})\int_{\eta(s)}^{s}(s-r)^{-\beta}b(Z_{\lfloor r/h\rfloor})dW_{r}dW_{s}.
\end{split}
\end{align}
Hence,
\begin{align*}
\mathbb{E}|Z_{n+1}|^{p}\leq& 5^{p-1}\mathbb{E}|Z_{0}|^{p}+5^{p-1}\mathbb{E}|\int_{0}^{t_{n+1}}(t_{n+1}-s)^{-\alpha}a(Z_{\lfloor s/h\rfloor})ds|^{p}\\
&+5^{p-1}C_{p}\mathbb{E}\bigg{(}\int_{0}^{t_{n+1}}(t_{n+1}-s)^{-2\beta}|b(Z_{\lfloor s/h\rfloor})|^{2}ds\bigg{)}^{\frac{p}{2}}\\
&+5^{p-1}C_{p}\mathbb{E}\bigg{(}\int_{0}^{t_{n+1}}(t_{n+1}-s)^{-2\beta}\big{(}\int_{0}^{\eta(s)}[(s-r)^{-\beta}-(\eta(s)-r)^{-\beta}]b(Z_{\lfloor r/h\rfloor})dW_{r}\big{)}^{2}ds\bigg{)}^{\frac{p}{2}}\\
&+5^{p-1}C_{p}\mathbb{E}\bigg{(}\int_{0}^{t_{n+1}}(t_{n+1}-s)^{-2\beta}\big{(}\int_{\eta(s)}^{s}(s-r)^{-\beta}b(Z_{\lfloor r/h\rfloor})dW_{r}\big{)}^{2}ds\bigg{)}^{\frac{p}{2}}.
\end{align*}
Using Jensen's inequality and Assumption \ref{assu1} give
\begin{align*}
&\sup\limits_{0\leq n\leq N-1}\mathbb{E}|Z_{n+1}|^{p}\\
&\leq C\mathbb{E}|Z_{0}|^{p}+C\left(\int_{0}^{t_{n+1}}(t_{n+1}-s)^{-\alpha}ds\right)^{p-1}
\int_{0}^{t_{n+1}}(t_{n+1}-s)^{-\alpha}\sup\limits_{0\leq s\leq t_{n+1}}(1+\mathbb{E}|Z_{\lfloor s/h\rfloor}|^{p})ds\\
&+C\left(\int_{0}^{t_{n+1}}(t_{n+1}-s)^{-2\beta}ds\right)^{\frac{p}{2}-1}\int_{0}^{t_{n+1}}(t_{n+1}-s)^{-2\beta}\sup\limits_{0\leq s\leq t_{n+1}}(1+\mathbb{E}|Z_{\lfloor s/h\rfloor}|^{p})ds\\
&+C\left(\int_{0}^{t_{n+1}}(t_{n+1}-s)^{-2\beta}ds\right)^{\frac{p}{2}-1}\int_{0}^{t_{n+1}}(t_{n+1}-s)^{-2\beta}\\
&\qquad  \cdot\mathbb{E}
\big{(}\int_{0}^{\eta(s)}[(s-r)^{-\beta}
-(\eta(s)-r)^{-\beta}]
b(Z_{\lfloor r/h\rfloor})dW_{r}\big{)}^{p}ds\\
&+C\left(\int_{0}^{t_{n+1}}(t_{n+1}-s)^{-2\beta}ds\right)^{\frac{p}{2}-1}\int_{0}^{t_{n+1}}(t_{n+1}-s)^{-2\beta}\mathbb{E}
\big{(}\int_{\eta(s)}^{s}(s-r)^{-\beta}
b(Z_{\lfloor r/h\rfloor})dW_{r}\big{)}^{p}ds
\end{align*}
Applying Burkholder-Davis-Gundy inequality, Jensen's inequality, Assumption \ref{assu1} and Lemma \ref{Lemma 3.4} to the above inequality yields
$$
\sup\limits_{0\leq n\leq N-1}\mathbb{E}|Z_{n+1}|^{p}\leq C.
$$
This completes the proof of the lemma.
\end{proof}
\begin{Lem}\label{Lemma 4.1}
Assume that the functions $a$ and $b$ are bounded, and their derivatives till third order are bounded.
Then for the local truncation error $R_{h}^{M}(t_{n+1})$, there is a constant $C$ such that for any $p\geq 1$
$$
\mathbb{E}|R_{h}^{M}(t_{n+1})|^{p}\leq \left\{
                                                     \begin{array}{ll}
                                                     Ch^{\min\{1-\alpha,1-2\beta\}p}, & \hbox{$\alpha\neq\frac{1}{2}$;} \\
                                                     %C\max\big{\{}h^{1-2\beta},h^{2(1-\beta)}\ln(1/h)\big{\}}, & \hbox{$\alpha=\frac{1}{2}$;} \\
                                                     C\max\{h^{\min\{1/2,1-2\beta\}p},h^{p(1-\beta)}(\ln(\frac{1}{h}))^{p/2}\}, & \hbox{$\alpha=\frac{1}{2}$.}
                                                     \end{array}
                                                          \right.
$$
\end{Lem}
\begin{proof} It follows from \eqref{aa11} that
$$
R_{h}^{M}(t_{n+1})=\tilde{\beta}+\tilde{\gamma}+\tilde{\delta}+R_{n}.
$$
Thus,
\begin{align}\label{q6}
\mathbb{E}|R_{h}^{M}(t_{n+1})|^{p}\leq 4^{p-1}\mathbb{E}|\tilde{\beta}|^{p}+4^{p-1}\mathbb{E}|\tilde{\gamma}|^{p}
+4^{p-1}\mathbb{E}|\tilde{\delta}|^{p}+4^{p-1}\mathbb{E}|R_{n}|^{p}.
\end{align}
Note that
\begin{eqnarray}\label{aa12}
\mathbb{E}|\tilde{\beta}|^{p}&\leq& C\mathbb{E}\big{|}\int_{0}^{t_{n+1}}(t_{n+1}-s)^{-\beta}\int_{0}^{\eta(s)}[(s-r)^{-\alpha}-(\eta(s)-r)^{-\alpha}]drdW_{s}\big{|}^{p}\nonumber\\
&&+C\mathbb{E}\big{|}\int_{0}^{t_{n+1}}(t_{n+1}-s)^{-\beta}\int_{\eta(s)}^{s}(s-r)^{-\alpha}drdW_{s}\big{|}^{p}\nonumber\\
&\leq& C\mathbb{E}\bigg{(}\int_{0}^{t_{n+1}}(t_{n+1}-s)^{-2\beta}\big{(}\int_{0}^{\eta(s)}
[(s-r)^{-\alpha}-(\eta(s)-r)^{-\alpha}]dr\big{)}^{2}ds\bigg{)}^{\frac p2}\nonumber\\
&&+C\mathbb{E}\bigg{(}\int_{0}^{t_{n+1}}(t_{n+1}-s)^{-2\beta}\big{(}
\int_{\eta(s)}^{s}(s-r)^{-\alpha}dr\big{)}^{2}ds\bigg{)}^{\frac p2}\nonumber\\
&\leq& C\left(\int_{0}^{t_{n+1}}(t_{n+1}-s)^{-2\beta}ds\right)^{\frac p2-1}\int_{0}^{t_{n+1}}(t_{n+1}-s)^{-2\beta}
\big{|}\int_{0}^{\eta(s)}[(s-r)^{-\alpha}-(\eta(s)-r)^{-\alpha}]dr\big{|}^{p}ds\nonumber\\
&&+C\left(\int_{0}^{t_{n+1}}(t_{n+1}-s)^{-2\beta}ds\right)^{\frac p2-1}\int_{0}^{t_{n+1}}(t_{n+1}-s)^{-2\beta}
\big{|}\int_{\eta(s)}^{s}(s-r)^{-\alpha}dr\big{|}^{p}ds\nonumber\\
&\leq& C\int_{0}^{t_{n+1}}(t_{n+1}-s)^{-2\beta}\bigg{(}\int_{0}^{\eta(s)}|(s-r)^{-\alpha}-(\eta(s)-r)^{-\alpha}|dr\bigg{)}^{p}ds
+Ch^{p(1-\alpha)}.
\end{eqnarray}
Obviously, there exists an integer $k\leq n$ such that $\eta(s)=kh$. Then
\begin{align*}
&\int_{0}^{\eta(s)}|(s-r)^{-\alpha}-(\eta(s)-r)^{-\alpha}|dr\\
&\leq \sum_{i=0}^{k-2}\int_{t_{i}}^{t_{i+1}}|(s-r)^{-\alpha}-(kh-r)^{-\alpha}|dr+\int_{(k-1)h}^{kh}(kh-r)^{-\alpha}ds\\
&\leq \sum_{i=0}^{k-2}\int_{t_{i}}^{t_{i+1}}(kh-t_{i+1})^{-\alpha}-((k+1)h-t_{i})^{-\alpha}dr+Ch^{1-\alpha}\\
&=h^{1-\alpha}\sum_{i=0}^{k-2}[(k-i-1)^{-\alpha}-(k+1-i)^{-\alpha}]+Ch^{1-\alpha}\\
&=h^{1-\alpha}(1+2^{-\alpha}-k^{-\alpha}-(k+1)^{-\alpha})+Ch^{1-\alpha}\leq Ch^{1-\alpha}.
\end{align*}
Combining the above results along with \eqref{aa12} yields
\begin{align}\label{q7}
\mathbb{E}|\tilde{\beta}|^{p}\leq Ch^{p(1-\alpha)}.
\end{align}
Exchanging  the order of integration in $\tilde{\gamma}$, we get
\begin{align*}
\tilde{\gamma}=&\sum_{i=0}^{n}\int_{0}^{t_{i}}\int_{t_{i}}^{t_{i+1}}a^{'}(X(t_{i}))(t_{n+1}-s)^{-\alpha}[(s-r)^{-\beta}-(t_{i}-r)^{-\beta}]ds
b(X(r))dW_{r}\\
&+\sum_{i=0}^{n}\int_{t_{i}}^{t_{i+1}}a^{'}(X(t_{i}))\int_{r}^{t_{i+1}}(s-r)^{-\beta}(t_{n+1}-s)^{-\alpha}dsb(X(r))dW_{r}.
\end{align*}
Therefore,
\begin{align}\label{aa13}
\begin{split}
\mathbb{E}|\tilde{\gamma}|^{p}=&\bigg{(}\int_{0}^{t_{n}}\big{|}\int_{\eta(r)+h}^{\eta(r)+2h}(t_{n+1}-s)^{-\alpha}
[(s-r)^{-\beta}-(\eta(s)-r)^{-\beta}]ds\big{|}^{2}dr\bigg{)}^{\frac p2}\\
&+\bigg{(}\int_{0}^{t_{n+1}}\big{|}\int_{r}^{\eta(r)+h}(s-r)^{-\beta}(t_{n+1}-s)^{-\alpha}ds\big{|}^{2}dr\bigg{)}^{\frac p2}\\
=&:I_{51}^{\frac{p}{2}}+I_{52}^{\frac{p}{2}}.
\end{split}
\end{align}
Note that
\begin{align*}
I_{51}=\sum_{l=0}^{n-1}\int_{t_{l}}^{t_{l+1}}\big{|}\int_{\eta(r)+h}^{\eta(r)+2h}(t_{n+1}-s)^{-\alpha}
[(s-r)^{-\beta}-(\eta(s)-r)^{-\beta}]ds\big{|}^{2}dr.
\end{align*}
For $l=n-1$, by H\"{o}lder inequality, we have
\begin{align*}
&\int_{t_{l}}^{t_{l+1}}\big{|}\int_{\eta(r)+h}^{\eta(r)+2h}(t_{n+1}-s)^{-\alpha}
[(s-r)^{-\beta}-(\eta(s)-r)^{-\beta}]ds\big{|}^{2}dr\\
&=\int_{t_{n-1}}^{t_{n}}\big{|}\int_{t_{n}}^{t_{n+1}}(t_{n+1}-s)^{-\alpha}
[(s-r)^{-\beta}-(t_{n}-r)^{-\beta}]ds\big{|}^{2}dr\\
&\leq \int_{t_{n-1}}^{t_{n}}\bigg{(}\int_{t_{n}}^{t_{n+1}}(t_{n+1}-s)^{-2\alpha}ds\bigg{)}
\bigg{(}\int_{t_{n}}^{t_{n+1}}[(t_{n}-r)^{-\beta}-(s-r)^{-\beta}]^{2}ds\bigg{)}dr\\
&\leq \int_{t_{n-1}}^{t_{n}}Ch^{1-2\alpha}\cdot h^{1-2\beta}dr\leq Ch^{3-2(\alpha+\beta)}.
\end{align*}
For $l<n-1$, we have
\begin{align*}
&\sum_{l=0}^{n-2}\int_{t_{l}}^{t_{l+1}}\big{|}\int_{\eta(r)+h}^{\eta(r)+2h}(t_{n+1}-s)^{-\alpha}
[(s-r)^{-\beta}-(\eta(s)-r)^{-\beta}]ds\big{|}^{2}dr\\
&=\sum_{l=0}^{n-2}\int_{t_{l}}^{t_{l+1}}\big{|}\int_{t_{l+1}}^{t_{l+2}}(t_{n+1}-t_{l+2})^{-\alpha}
[(\eta(s)-r)^{-\beta}-(s-r)^{-\beta}]ds\big{|}^{2}dr\\
&\leq \sum_{l=0}^{n-2}\int_{t_{l}}^{t_{l+1}}\big{|}\int_{t_{l+1}}^{t_{l+2}}(t_{n+1}-t_{l+2})^{-\alpha}
[(t_{l+1}-r)^{-\beta}-(t_{l+2}-t_{l})^{-\beta}]ds\big{|}^{2}dr\\
&=h^{2(1-\alpha)}\sum_{l=0}^{n-2}(n-l-1)^{-2\alpha}\int_{t_{l}}^{t_{l+1}}[(t_{l+1}-r)^{-\beta}-(2h)^{-\beta}]^{2}dr\\
&\leq Ch^{2(1-\alpha)}\cdot h^{1-2\beta}\sum_{l=0}^{n-2}(n-l-1)^{-2\alpha}.
\end{align*}
Thus,
\begin{align*}
&\sum_{l=0}^{n-2}\int_{t_{l}}^{t_{l+1}}\big{|}\int_{\eta(r)+h}^{\eta(r)+2h}(t_{n+1}-s)^{-\alpha}
[(s-r)^{-\beta}-(\eta(s)-r)^{-\beta}]ds\big{|}^{2}dr\nonumber\\
&\qquad\qquad \qquad \leq \left\{
 \begin{array}{ll}  Ch^{3-2(\alpha+\beta)}, & \hbox{$\alpha>\frac{1}{2}$;} \\
                                                              Ch^{2(1-\beta)}\ln(1/h), & \hbox{$\alpha=\frac{1}{2}$;} \\
                                                              Ch^{2(1-\beta)}, & \hbox{$\alpha<\frac{1}{2}$.}
                                                            \end{array}  \right.
\end{align*}
Consequently,
\begin{align*}
I_{51}\leq \left\{
                                                            \begin{array}{ll}
                                                              Ch^{3-2(\alpha+\beta)}, & \hbox{$\alpha>\frac{1}{2}$;} \\
                                                              Ch^{2(1-\beta)}\ln(1/h), & \hbox{$\alpha=\frac{1}{2}$;} \\
                                                              Ch^{2(1-\beta)}, & \hbox{$\alpha<\frac{1}{2}$.}
                                                            \end{array}
                                                          \right.
\end{align*}
 In a similar way, we can prove that
\begin{align*}
I_{52}\leq \left\{
                                                            \begin{array}{ll}
                                                              Ch^{3-2(\alpha+\beta)}, & \hbox{$\alpha>\frac{1}{2}$;} \\
                                                              Ch^{2(1-\beta)}\ln(1/h), & \hbox{$\alpha=\frac{1}{2}$;} \\
                                                              Ch^{2(1-\beta)}, & \hbox{$\alpha<\frac{1}{2}$.}
                                                            \end{array}
                                                          \right.
\end{align*}
%\begin{align}
%\left\{
%  \begin{array}{ll}
%    1, & \hbox{1;} \\
%    1, & \hbox{2;} \\
%    1, & \hbox{3.}
%  \end{array}
%\right.
%\end{align}
%There exists an integer $k\leq n+1$ such that $\eta(s)=kh$, then
%\begin{align*}
%&\int_{0}^{\eta(s)}|(s-r)^{-\alpha}-(\eta(s)-r)^{-\alpha}|^{2}dr\\
%&=\sum_{i=0}^{k-2}\int_{t_{i}}^{t_{i+1}}|\alpha\int_{\eta(s)}^{s}(\tau-r)^{-\alpha-1}d\tau|^{2}dr
%+\int_{t_{k-1}}^{t_{k}}|(s-r)^{-\alpha}-(kh-r)^{-\alpha}|^{2}dr\\
%&\leq C\sum_{i=0}^{k-2}\int_{t_{i}}^{t_{i+1}}|\int_{\eta(s)}^{s}(kh-t_{i+1})^{-\alpha-1}d\tau|^{2}dr+\int_{t_{k-1}}^{t_{k}}(s-r)^{-2\alpha}dr
%+\int_{t_{k-1}}^{t_{k}}(kh-r)^{-2\alpha}dr\\
%&\leq C\sum_{i=0}^{k-2}\int_{t_{i}}^{t_{i+1}}(n-i-1)^{-2(\alpha+1)}h^{-2\alpha}dr+Ch^{1-2\alpha}\leq Ch^{1-2\alpha}.
%\end{align*}
Combining the above results with \eqref{aa13}, we arrive at
\begin{align}\label{q8}
\mathbb{E}|\tilde{\gamma}|^{p}\leq \left\{
                                                            \begin{array}{ll}
                                                              Ch^{(3-2(\alpha+\beta))p/2}, & \hbox{$\alpha>\frac{1}{2}$;} \\
                                                              Ch^{p(1-\beta)}(\ln(1/h))^{p/2}, & \hbox{$\alpha=\frac{1}{2}$;} \\
                                                              Ch^{p(1-\beta)}, & \hbox{$\alpha<\frac{1}{2}$.}
                                                            \end{array}
                                                          \right.
\end{align}
Moreover,
\begin{align}\label{q9}
\begin{split}
\mathbb{E}|\tilde{\delta}|^{p}=&\mathbb{E}\big{|}\int_{0}^{t_{n+1}}(t_{n+1}-s)^{-\beta}
\big{(}\int_{0}^{\eta(s)}[(s-r)^{-\beta}-(\eta(s)-r)^{-\beta}]dW_{r}
+\int_{\eta(s)}^{s}(s-r)^{-\beta}dW_{r}\big{)}^{2}dW_{s}\big{|}^{p}\\
\leq& 2^{p}\mathbb{E}\bigg{(}\int_{0}^{t_{n+1}}(t_{n+1}-s)^{-2\beta}\big{(}\int_{0}^{\eta(s)}
[(s-r)^{-\beta}-(\eta(s)-r)^{-\beta}]dW_{r}\big{)}^{4}ds\bigg{)}
^{\frac{p}{2}}\\
&+2^{p}\mathbb{E}\bigg{(}\int_{0}^{t_{n+1}}(t_{n+1}-s)^{-2\beta}
\big{(}\int_{\eta(s)}^{s}(s-r)^{-\beta}dW_{r}\big{)}^{4}ds\bigg{)}^{\frac{p}{2}}\\
\leq& C\left(\int_{0}^{t_{n+1}}(t_{n+1}-s)^{-2\beta}ds\right)^{\frac{p}{2}-1}\int_{0}^{t_{n+1}}(t_{n+1}-s)^{-2\beta}\\
&\cdot\mathbb{E}\big{(}\int_{0}^{\eta(s)}|(s-r)^{-\beta}-(\eta(s)-r)^{-\beta}|dW_{r}\big{)}^{2p}ds\\
&+C\left(\int_{0}^{t_{n+1}}(t_{n+1}-s)^{-2\beta}ds\right)^{\frac{p}{2}-1}\int_{0}^{t_{n+1}}(t_{n+1}-s)^{-2\beta}
\mathbb{E}\big{(}\int_{\eta(s)}^{s}(s-r)^{-\beta}dW_{r}\big{)}^{2p}ds\\
\leq& C\int_{0}^{t_{n+1}}(t_{n+1}-s)^{-2\beta}
\big{(}\int_{0}^{\eta(s)}|(s-r)^{-\beta}-(\eta(s)-r)^{-\beta}|^{2}dr\big{)}^{p}ds\\
&+\int_{0}^{t_{n+1}}(t_{n+1}-s)^{-2\beta}
\big{(}\int_{\eta(s)}^{s}(s-r)^{-2\beta}dr\big{)}^{p}ds\\
\leq& Ch^{(1-2\beta)p}\,.
\end{split}
\end{align}
%where $\eta(s)=\lfloor s/h\rfloor h$, $\lfloor x\rfloor$
%  is the nearest integer less than or equal to $x$.
%The above estimation is too rough, can we have more precise estimation, so that we can get the order of convergence $O(h^{2(1-2\alpha)})$?
%If we can use the BDG inequality for the integrand inside, then
%\begin{align}
%\begin{split}
%\mathbb{E}|\tilde{\delta}|^{2}=&\mathbb{E}\big{|}\int_{0}^{t_{n+1}}(t_{n+1}-s)^{-\alpha}\mathbb{E}\big{(}\int_{0}^{\eta(s)}[(s-r)^{-\alpha}-(\eta(s)-r)^{-\alpha}]dW_{r}
%+\int_{\eta(s)}^{s}(s-r)^{-\alpha}dW_{r}\big{)}^{2}dW_{s}\big{|}^{2}\\
%\leq& 2\int_{0}^{t_{n+1}}(t_{n+1}-s)^{-2\alpha}\big{(}\int_{0}^{\eta(s)}[(s-r)^{-\alpha}-(\eta(s)-r)^{-\alpha}]^{2}dr\big{)}^{2}ds\\
%&+2\int_{0}^{t_{n+1}}(t_{n+1}-s)^{-2\alpha}\big{(}\int_{\eta(s)}^{s}(s-r)^{-2\alpha}dr\big{)}^{2}ds\\
%\leq&  Ch^{2(1-2\alpha)}.
%\end{split}
%\end{align}}
Using an analogous  technique, it is easy to verify that $\mathbb{E}|R_{n}|^{2}$ has a higher order with respect to stepsize $h$. Hence, the final results follows from \eqref{q6}, \eqref{q7}, \eqref{q8} and \eqref{q9}.
\end{proof}

\begin{proof}[Proof of Theorem \ref{t.2.4}]
It follows from \eqref{a211} and \eqref{aa11} that
\begin{eqnarray}
&&\varepsilon_{n+1}\\
&&=R_{h}^{M}(t_{n+1})+\int_{0}^{t_{n+1}}(t_{n+1}-s)^{-\alpha}(a(X({\eta(s)}))-a(Z_{\lfloor s/h\rfloor}))ds\nonumber\\
&&+\int_{0}^{t_{n+1}}(t_{n+1}-s)^{-\beta}(b(X({\eta(s)}))-b(Z_{\lfloor s/h\rfloor}))dW_{s}\nonumber\\
&&+\int_{0}^{t_{n+1}}(t_{n+1}-s)^{-\beta}b^{'}(X({\eta(s)}))\bigg{\{}\int_{0}^{\eta(s)}[(s-r)^{-\beta}-(\eta(s)-r)^{-\beta}]
(b(X(r))-b(X({\eta(r)})))dW_{r}\bigg{\}}dW_{s}\nonumber\\
&&+\int_{0}^{t_{n+1}}(t_{n+1}-s)^{-\beta}b^{'}(X({\eta(s)}))\bigg{\{}
\int_{\eta(s)}^{s}(s-r)^{-\beta}(b(X(r))-b(X({\eta(s)})))dW_{r}\bigg{\}}dW_{s}\nonumber\\
&&+\int_{0}^{t_{n+1}}(t_{n+1}-s)^{-\beta}b^{'}(X({\eta(s)}))\bigg{\{}\int_{0}^{\eta(s)}[(s-r)^{-\beta}-(\eta(s)-r)^{-\beta}]
(b(X({\eta(r)}))-b(Z_{\lfloor r/h\rfloor}))dW_{r}\bigg{\}}dW_{s}\nonumber\\
&&+\int_{0}^{t_{n+1}}(t_{n+1}-s)^{-\beta}b^{'}(X({\eta(s)}))\bigg{\{}\int_{\eta(s)}^{s}(s-r)^{-\beta}(b(X({\eta(s)}))-b(Z_{\lfloor r/h\rfloor}))dW_{r}\bigg{\}}dW_{s}\nonumber\\
&&+\int_{0}^{t_{n+1}}(t_{n+1}-s)^{-\beta}(b^{'}(X({\eta(s)}))-b^{'}(Z_{\lfloor s/h\rfloor}))
\bigg{\{}\int_{0}^{\eta(s)}[(s-r)^{-\beta}-(\eta(s)-r)^{-\beta}]b(Z_{\lfloor r/h\rfloor})dW_{r}\bigg{\}}dW_{s}\nonumber\\
&&+\int_{0}^{t_{n+1}}(t_{n+1}-s)^{-\beta}(b^{'}(X({\eta(s)}))-b^{'}(Z_{\lfloor s/h\rfloor}))
\bigg{\{}\int_{\eta(s)}^{s}(s-r)^{-\beta}b(Z_{\lfloor r/h\rfloor})dW_{r}\bigg{\}}dW_{s}.
\end{eqnarray}
From the proof of Lemma \ref{Lemma 4.1}, it is easy to
verify that
\begin{eqnarray*}
&&\mathbb{E}\big{|}\int_{0}^{t_{n+1}}(t_{n+1}-s)^{-\beta}b^{'}(X({\eta(s)})) \bigg{\{}\int_{0}^{\eta(s)}[(s-r)^{-\beta}-(\eta(s)-r)^{-\beta}]\\
&&\qquad\qquad \cdot
(b(X(r))-b(X({\eta(r)})))dW_{r}\bigg{\}}dW_{s}\big{|}^{p}\leq Ch^{p(1-2\beta)},
\end{eqnarray*}
and
$$
\mathbb{E}\big{|}\int_{0}^{t_{n+1}}(t_{n+1}-s)^{-\beta}b^{'}(X({\eta(s)}))\bigg{\{}\int_{\eta(s)}^{s}(s-r)^{-\beta}(b(X(r))-b(X({\eta(s)})))dW_{r}
\bigg{\}}dW_{s}\big{|}^{p}\leq Ch^{p(1-2\beta)}.
$$
Consequently,
\begin{align*}
\mathbb{E}|\varepsilon_{n+1}|^{p}\leq C\bigg{\{}\mathbb{E}|R_{h}^{M}(t_{n+1})|^{p}+h^{p(1-2\beta)}+h^{1-2\beta}\sum_{i=0}^{n}(n+1-i)^{-2\beta}
\max\limits_{1\leq i\leq n}\mathbb{E}|\varepsilon_{i}|^{p}\bigg{\}}.
\end{align*}
Thus,
\begin{align*}
\max\limits_{0\leq n\leq N-1}\mathbb{E}|\varepsilon_{n+1}|^{p}\leq C\bigg{\{}\mathbb{E}|R_{h}^{M}(t_{n+1})|^{p}+h^{p(1-2\beta)}+h^{1-2\beta}\sum_{i=0}^{n}(n+1-i)^{-2\beta}
\max\limits_{0\leq i\leq n}\mathbb{E}|\varepsilon_{i}|^{p}\bigg{\}},
\end{align*}
where Lemma \ref{lemma 4.1} and Assumption \ref{assu2} are used.
Furthermore, by  Lemma \ref{Lemma 3.2} and Lemma \ref{Lemma 4.1}
$$
\max\limits_{0\leq n\leq N-1}\mathbb{E}|\varepsilon_{n+1}|^{p}\leq \left\{
                                                     \begin{array}{ll}
                                                     Ch^{\min\{1-\alpha,1-2\beta\}p}, & \hbox{$\alpha\neq\frac{1}{2}$;} \\
                                                     %C\max\big{\{}h^{1-2\beta},h^{2(1-\beta)}\ln(1/h)\big{\}}, & \hbox{$\alpha=\frac{1}{2}$;} \\
                                                     C\max\{h^{\min\{\frac p2,p(1-2\beta)\}},h^{p(1-\beta)}(\ln(\frac{1}{h}))^{\frac p2}\}, & \hbox{$\alpha=\frac{1}{2}$, }
                                                     \end{array}
                                                          \right.
$$
 which completes the proof.
 \end{proof}

As   consequences of the    Theorems \ref{theorem2.3} and \ref{t.2.4}   we have the following corollaries.
\begin{Coro}  The rates  of convergence in $p$-th moment ($p\geq 1$) of $\theta$-Euler-Maruyama  and Milstein-type schemes for the following
 It\^{o}-Doob stochastic fractional differential equations
$$
X(t)=X_{0}+\int_{0}^{t}b(X(s))ds+\int_{0}^{t}\sigma_{1}(X(s))dW(s)+\alpha\int_{0}^{t}\frac{\sigma_{2}(X(s))}{(t-s)^{1-\alpha}}ds,
~~0<\alpha<1,
$$
(which is   studied in \cite{Abouagwa2019Doob})   are
 $\min\{\frac{1}{2},\alpha\}$ and $\min\{1,2\alpha\}$, respectively.
\end{Coro}

\begin{Coro}
 Caputo fractional stochastic differential equation
$$
{}^{C}D_{0+}^{\alpha}X(t)=b(X(t))+\sigma(X(t))\frac{dW_{t}}{dt},~~\alpha>\frac{1}{2},
$$
 was  studied in \cite{Son2018Caputo,Anh2019variation}, which is equivalent to SVIEs with weakly singular kernel of the form
$$
X(t)=X_{0}+\frac{1}{\Gamma(\alpha)}\bigg{(}\int_{0}^{t}(t-s)^{\alpha-1}b(X(s))ds+\int_{0}^{t}(t-s)^{\alpha-1}\sigma(X(s))dW_{s}\bigg{)}.
$$
 If we apply our $\theta$-Euler-Maruyama  and Milstein-type schemes  to this equation then    the rates  of convergence in $p$-th moment ($p\geq 1$)  are
 $\alpha-\frac{1}{2}$ and $2\alpha-1$, respectively.
\end{Coro}

%\begin{Rem}\label{Remark 4.1}
%{\color{blue}The assumptions $\sup\limits_{t\in[-\tau,T]}\|Y(t)\|_{L^{4q-2}(\Omega;\mathbb{R}^{d})}<\infty$ and $\sup\limits_{t\in[-\tau,T]}\|Y(t)\|_{L^{6q-4}(\Omega;\mathbb{R}^{d})}<\infty$ are necessary in our proof. Since $q\in [1,\infty)$,
% it is obvious that $6q-4\geq 4q-2$. If $p\geq 6q-4$ is satisfied, then it follows from Lemma \ref{Lemma 2.1} that $\sup\limits_{t\in[-\tau,T]}\|Y(t)\|_{L^{6q-4}(\Omega;\mathbb{R}^{d})}<\infty$.}
%\end{Rem}
\begin{figure}[!ht]
\centering
\includegraphics[height=8cm,width=12cm]{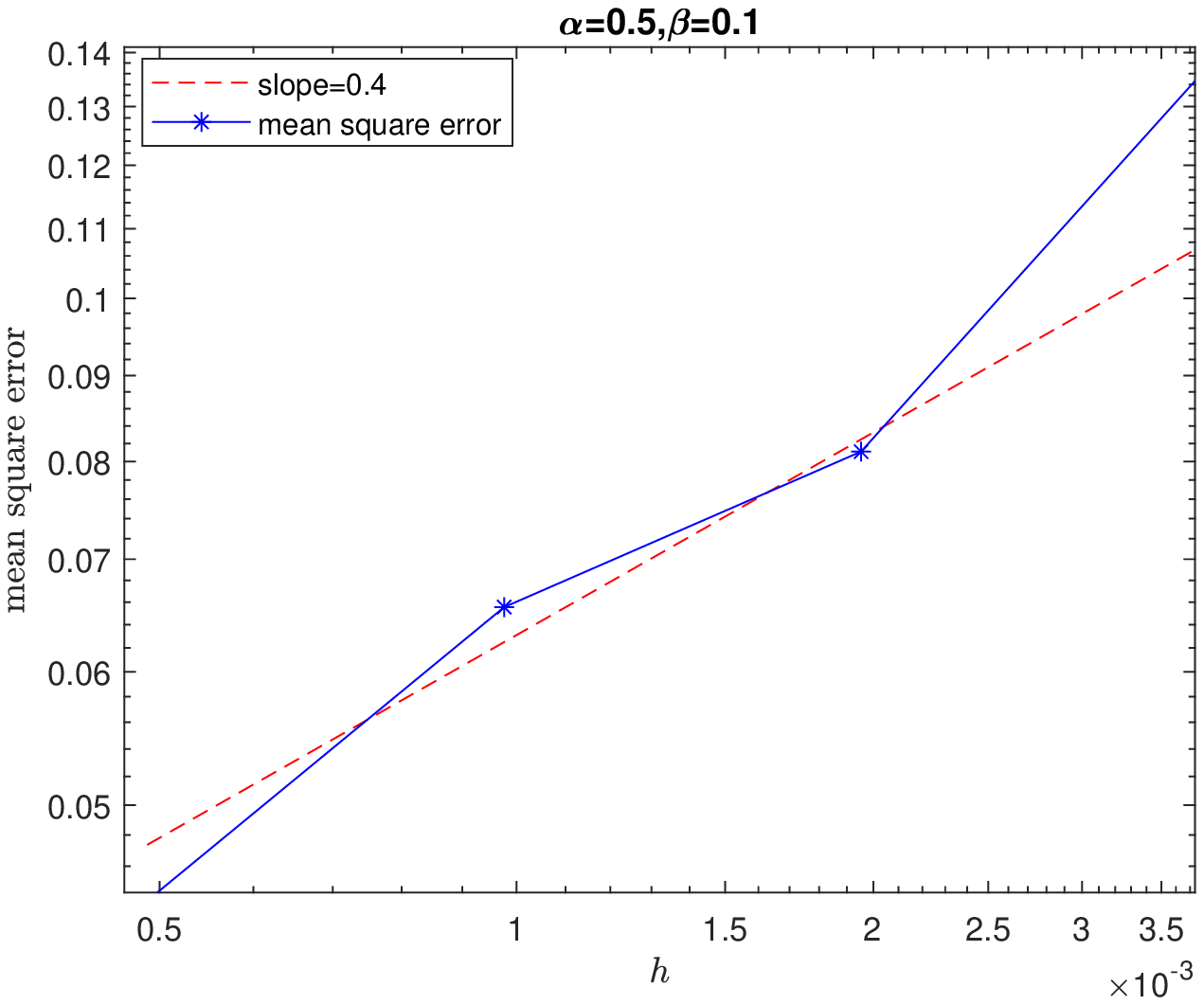}
\caption{Mean square error of the $\theta$-EM method with $\theta=0.5$ for \eqref{ex1}.}
\label{1}
\end{figure}

\begin{figure}[!ht]
\centering
\includegraphics[height=8cm,width=12cm]{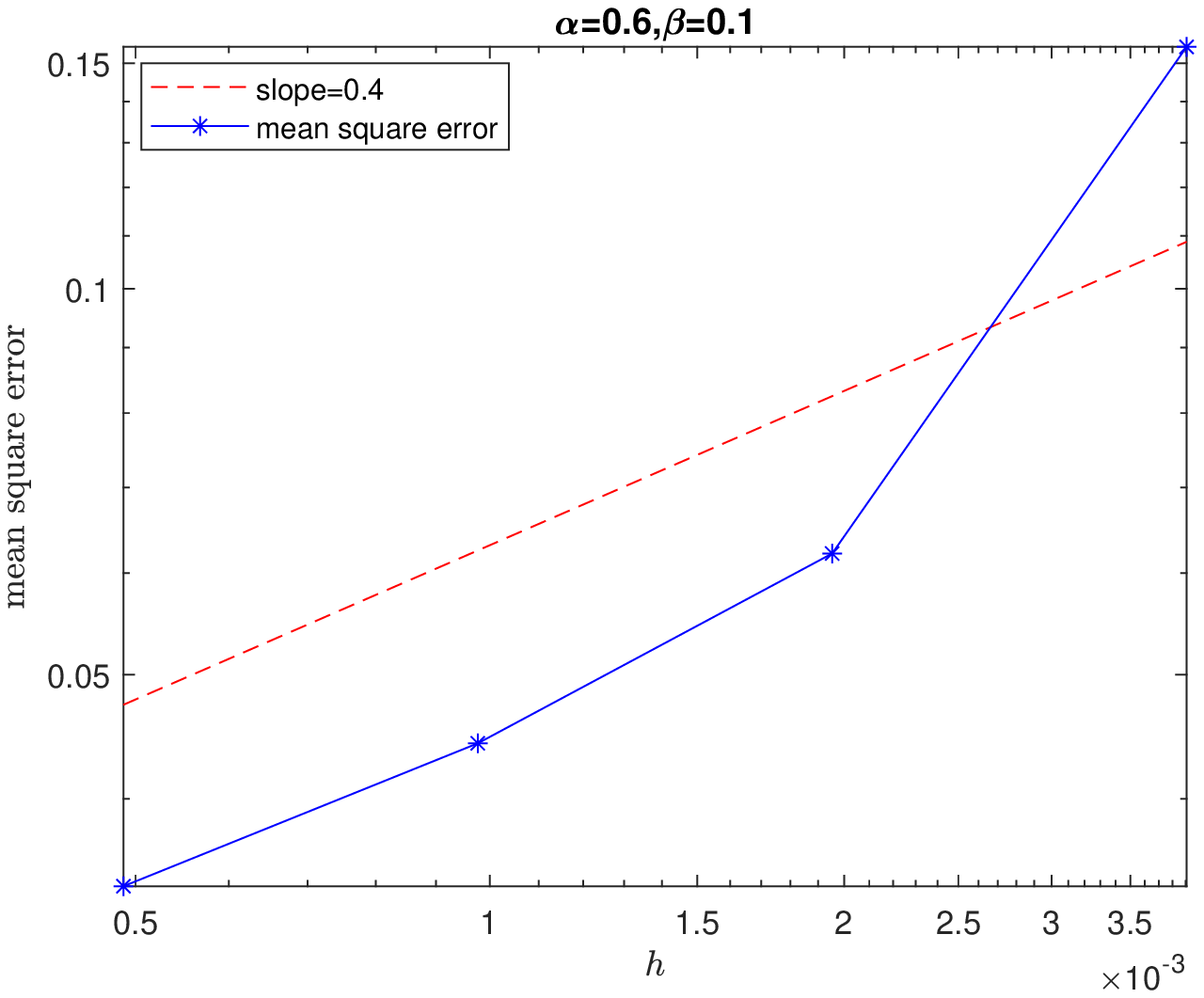}
\caption{Mean square error of the $\theta$-EM method with $\theta=0.5$ for \eqref{ex1}.}
\label{2}
\end{figure}

\begin{figure}[!ht]
\centering
\includegraphics[height=8cm,width=12cm]{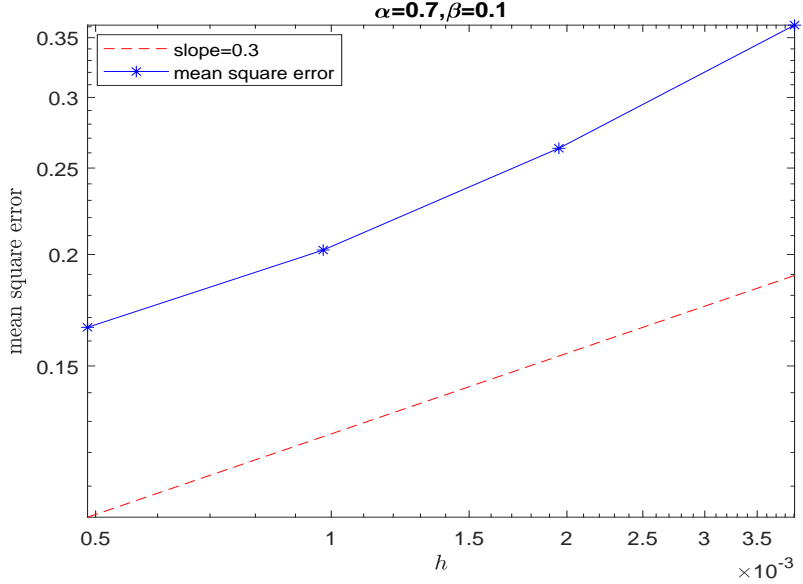}
\caption{Mean square error of the $\theta$-EM method with $\theta=1$ for \eqref{ex1}.}
\label{3}
\end{figure}
\section{Numerical experiments}
\begin{example}\label{example1}
We consider the following example
\begin{equation}\label{ex1}
X(t)=1+\int_{0}^{t}(t-s)^{-\alpha}\sin(X(s))ds+\frac{1}{2}\int_{0}^{t}(t-s)^{-\beta}(\cos(X(s))+2)dW_{s},~~~t\in [0, 1].
\end{equation}
\end{example}
Due to appearance of the  singularity in the
above stochastic integral, it is  difficult for us to illustrate  the convergence rate of the Milstein-type scheme. Here, we only check    the order of convergence of the $\theta$-EM scheme numerically. We regard the numerical solution yielded by small stepsize $h^{*}=2^{-13}$
 as 'exact' solution. Moreover, the corresponding numerical solutions are generated by four different stepsizes
  $h=4h^{*}, 8h^{*}, 16h^{*}$ and $32h^{*}$, respectively. The mean square errors of $\theta$-EM
  scheme  are calculated at the terminal time $t_{N}=T=1$ by
  $$
  e=\bigg{(}\frac{1}{M}\sum_{i=1}^M |X_{T}^{(i)}-X_{N}^{(i)}|^{2}\bigg{)}^{\frac 12},
  $$
   where the expectation is approximated by averaging over $M=1000$ Brownian sample paths. The mean square errors are plotted in Figs. \ref{1},
   \ref{2}, \ref{3}  in log-log scale. In these plots, the reference lines and error lines are parallel to each other,  revealing the
   convergence rate of $\theta$-EM scheme is $\min\{1-\alpha,\frac{1}{2}-\beta\}$.
\section{Conclusion}
Our aim in this work is to investigate a $\theta$-Euler-Maruyama scheme and a Milstein type scheme
for SVIEs with weakly singular kernels.
 Since It\^{o} formula is not available, the classical  proof techniques are no longer used.
 Our new strategy is based on the Taylor formula, classical fractional calculus, and discrete and
continuous typed Gronwall inequalities with weakly singular kernels. The convergence rates
of these schemes have been given by a technical
  analysis on   the equation the solution satisfies. And the convergence results of $\theta$-Euler-Maruyama
  scheme are demonstrated through some numerical experiments. In forthcoming works, we study a Milstein-
type method for SVIEs with diagonal and boundary singularities of the kernel (cf. \cite{Pedas2006,Kolk2009,Dai2020doubly}).
  Our future work is to verify
 whether the order of convergence is optimal. In addition, we will study how to effectively model
 the multiple stochastic singular integrals.
\bibliographystyle{elsarticle-num}
\bibliography{reference}
\end{document}